\documentclass[12pt,leqno,a4paper]{amsart}
\usepackage[top=3cm, bottom=2.5cm, left=2.5cm, right=2.5cm]{geometry}
\usepackage{amssymb,amsmath,amscd,amsbsy,graphicx,fontenc,amsthm,mathrsfs,mathtools}
\usepackage[frame, cmtip, arrow, matrix, line, graph, curve]{xy}
\usepackage{graphpap, color}
\usepackage{ifthen}
\usepackage{relsize}
\usepackage[bookmarks,breaklinks,colorlinks=true,linkcolor=blue,citecolor=blue,urlcolor=blue,menucolor=blue,linktoc=all,pdftex]{hyperref}

\usepackage{multirow}

\usepackage[right,displaymath,mathlines]{lineno}

\usepackage{fancyhdr}
\usepackage{indentfirst}
\usepackage{paralist}
\usepackage{pstricks}
\usepackage{bbm}
\usepackage{tikz}

\makeatletter
\def\th@plain{%
    \thm@notefont{}
    \itshape
}
\def\th@definition{%
    \thm@notefont{}
    \normalfont
}
\makeatother

\newtheorem{theorem}{Theorem}[section]

\newtheorem{lemma}[theorem]{Lemma}
\newtheorem{proposition}[theorem]{Proposition}
\newtheorem{corollary}[theorem]{Corollary}
\newtheorem{notation}[theorem]{Notation}

\newtheorem*{theorem*}{Theorem}

\theoremstyle{remark}
\newtheorem{remark}[theorem]{Remark}

\numberwithin{equation}{section}
\newcommand{\EE}{{\mathbb{E}}}
\newcommand{\NN}{{\mathbb{N}}}
\newcommand{\ZZ}{{\mathbb{Z}}}

\newcommand{\RR}{{\mathbb{R}}}

\newcommand{\QQ}{{\mathbb{Q}}}

\DeclareMathOperator{\rk}{rk}

\DeclareMathOperator{\vol}{vol}

\DeclareMathOperator{\Prob}{Prob}

\newcommand{\diag}{\mathrm{diag}}

\newcommand{\tp}[1]{^\mathrm{t}{#1}}
\newcommand{\SL}{\mathrm{SL}}

\newcommand{\ASL}{\mathrm{ASL}}

\newcommand{\Id}{\mathrm{Id}}

\renewcommand{\d}{\hspace{0.5mm} \mathrm d}
\newcommand{\bm}{\mathbf}
\newcommand{\mychi}{\raisebox{2pt}{$\chi$}}

\newcommand{\vx}{\mathbf x}
\newcommand{\vy}{\mathbf y}
\newcommand{\vw}{\mathbf w}
\newcommand{\vv}{\mathbf v}

\newcommand{\vb}{\mathbf b}

\newcommand{\vl}{\ensuremath{\boldsymbol\ell}}
\providecommand{\ve}{\mathbf{ e}}
\providecommand{\vp}{\mathbf{p}}

\newcommand{\vm}{\mathbf m}
\newcommand{\vn}{\mathbf n}
\newcommand{\vk}{\mathbf k}
\newcommand{\vz}{\mathbf z}

\newcommand{\Mat}{\mathrm{Mat}}

\newcommand{\Siegel}[2]{\mathcal S_{#1}(#2)}

\renewcommand{\varpi}{\pi}

\newcommand{\origin}{\mathbf 0}

\newcommand{\ds}{\diamondsuit}

\definecolor{cmd}{rgb}{1.0, 0.35, 0.21}
\allowdisplaybreaks

\begin{document}

\title[Affine and congruence point counting]{Higher moment formulae and limiting distributions of lattice points}

\author{Mahbub Alam}
\address{\textbf{Mahbub Alam} \\
    Department of Mathematics,
    Uppsala University, Sweden \\
\url{https://sites.google.com/view/mahbubweb}}
\email{mahbub.dta@gmail.com, mahbub.alam@math.uu.se}

\author{Anish Ghosh}
\address{\textbf{Anish Ghosh} \\
School of Mathematics, Tata Institute of Fundamental Research, Homi Bhabha Road, Colaba, Mumbai, India 400005}
\email{ghosh@math.tifr.res.in}

\author{Jiyoung Han}
\address{\textbf{Jiyoung Han} \\ Korea Institute for Advanced Study (KIAS), Seoul, Republic of Korea}
\email{hanjiwind@gmail.com} 


\begin{abstract}
    We establish higher moment formulae for Siegel transforms on the space of affine unimodular lattices as well as on certain congruence quotients of $\SL_d(\RR)$. As applications, we prove functional central limit theorems for lattice point counting for affine and congruence lattices using the method of moments. 
\end{abstract}

\maketitle

\tableofcontents

\section{Introduction}
 Let $X_d$ denote the space of unimodular lattices in $\RR^d$ which can be naturally identified with $\SL_d(\ZZ)\backslash \SL_d(\RR)$ and denote by $\mu$ the Haar measure on $X_d$ normalized to be a probability measure. Let $f:\RR^d \rightarrow \RR$ be a bounded function of compact support. The Siegel transform $\Siegel{1}{f}$ of $f$ is defined by
\[
    \Siegel{1}{f}(\Lambda)=\sum_{\vm \in \Lambda}
    f(\vm),\; \Lambda\in \SL_d(\ZZ)\backslash \SL_d(\RR).
\]
In \cite{Sie45}, Siegel proved that 
$$ \int_{X_d}\Siegel{1}{f}d\mu = \int_{\RR^d}f(x)dx+f(\origin).$$

This result, often referred to as \emph{Siegel's mean value formula} is a fundamental result in the geometry of numbers and has proved to be indispensable in homogeneous dynamics, especially in applications to Diophantine problems. Following Siegel's result, Rogers \cite{Rogers55} established intricate formulae for the higher moments of Siegel transforms (see Theorem \ref{Rogers moment} in Section 2). These formulae have since become an important tool in a wide variety of Diophantine problems. It is of considerable interest to prove analogues of Siegel's and Rogers's formulae for other homogeneous spaces. In this paper, we will establish explicit higher moment formulae for analogues of the Siegel transform on the following two homogeneous spaces, which are equipped with natural invariant probability measures $\mu_Y$ and $\mu_{q}$ on $Y$ and $Y_{\vp/q}$, respectively (see Section~\ref{Section:Higher Moment Formulae}). 
\begin{itemize}
\item The space $Y := \ASL_d(\ZZ)\backslash \ASL_d(\RR)$.
\item The space $Y_{{\vp}/q} := \left\{\left(\ZZ^d+\dfrac {\vp} q\right)g : g\in \SL_d(\RR) \right\}$, 
where $\vp\in \ZZ^d \smallsetminus \{\origin\}$ and $q\in \NN_{\ge 2}$ with $\gcd(\vp, q)=1$.
\end{itemize}

There have been many developments since Rogers's work; among those pertinent to the present paper, there is the recent work \cite{Han} of the third named author where $S$-arithmetic versions of Rogers's theorems are established. Analogues of Siegel transforms for $Y$ and $Y_{{\vp}/q}$ have been considered and in fact a second moment formula has been obtained in each case. In the affine case, by El-Baz, Marklof and Vinogradov \cite{EMV2015} where they were used to study the distribution of gaps between lattice directions (see also \cite{Ath}); in the congruence case, by Ghosh, Kelmer and Yu \cite{GKY2020} where they were used to study effective versions of an inhomogeneous version of Oppenheim's conjecture on quadratic forms. In fact, they have other applications as well, we refer the reader to \cite{AGY} for an application of the congruence second moment formula to Diophantine approximation and to \cite{GH} for an $S$-arithmetic version of the congruence second moment formula with applications to quadratic forms. 

The main results in the present paper are formulas computing all the higher moments of Siegel transforms for both the affine and congruence cases. We also obtain analogues of a modification to Rogers's formula, due to Str\"{o}mbergsson and  S\"{o}dergren \cite{StSo}. Our proof of the higher moment formulae owes a lot to the breakthrough work of Marklof and Str\"{o}mbergsson \cite{MS2010}. As will become clear, we make significant use of the ideas in Section $7$ of their paper. Our formulas are explicit but, as is the case with Rogers's original formula, are heavy on notation and need some buildup to state. We therefore postpone stating them to the next section. The reader will find the higher moment formula for Siegel transforms on $Y$ in Theorem \ref{higher moment formula: affine}, and the formula for Siegel transforms on $Y_{{\vp}/q}$ in Theorem \ref{higher moment formula}.  If history is a reliable guide, then our higher moment formulae will find good uses in counting problems. In the present paper, we provide applications to limiting distributions in lattice point counting problems. We devote the remainder of the introduction to discussing these applications.

\subsection{Applications to counting results}
Our counting results are inspired by the work of Str\"{o}mbergsson and  S\"{o}dergren \cite{StSo}. Given $d \geq 2$, a lattice $L \in X_d$ and a real number $x \geq 0$, set
$$ N_{d, L}(x):= \#\left\{m \in L \backslash \{0\}~:~|m| \leq \left(\frac{x}{V_d}\right)^{1/d}\right\}, $$
where $V_d$ denotes the volume of the unit ball in $\RR^d$. Further let 
$$R_{d, L}(x) := N_{d, L}(x) - x $$
be the error term in the Gauss circle problem. Str\"{o}mbergsson and  S\"{o}dergren proved several interesting results regarding the behaviour of $R_{n, L}$ including the following central limit theorem for a random lattice $L$. 
\begin{theorem*}[Str\"{o}mbergsson and  S\"{o}dergren \cite{StSo}]\label{S-S}
Let $\phi : \ZZ_{+} \to \RR_{+}$ be any function satisfying $\lim_{n \to \infty} \phi(d) = \infty$ and $\phi(d) = O_{\varepsilon}(e^{\varepsilon d})$ for every $\varepsilon > 0$. Let $Z^{(B)}_d$ be the random variable
\[
    Z^{(B)}_d := \frac{1}{\sqrt{2\phi(d)}}R_{d, L}(\phi(d))
\]
with $L$ picked at random in $(X_d, \mu)$.
Then
\[
    Z^{(B)}_d \rightarrow \mathcal N(0,1) \text{ as } d \to \infty
\]
in distribution.
\end{theorem*}
Earlier, S\"{o}degren \cite{Sod} studied the distribution of lengths of lattice vectors in a random lattice of large dimension. Str\"{o}mbergsson and  S\"{o}dergren used the central limit theorem above in conjunction with S\"{o}degren's theorem, to establish the following theorem indicating Poissonian behaviour for sequences growing sub-exponentially with respect to the dimension.
\begin{theorem*}[Str\"{o}mbergsson and  S\"{o}dergren \cite{StSo}]
Let $\phi : \ZZ_{+} \to \RR_{+}$ be any function satisfying $\lim_{n \to \infty} \phi(d) = \infty$ and
$\phi(d) = O_{\varepsilon}(e^{\varepsilon d})$
for every $\varepsilon > 0$. Let $\mathcal{N}(x)$ be a Poisson distributed random variable with expectation $x/2$. Then
$$\Prob_{\mu}(N_{d, L}(x) \leq 2N) - \Prob(\mathcal{N}(x) \leq N) \to 0 \text{ as } d \to \infty,  $$
uniformly with respect to all $N, x \geq 0$ satisfying $\min(x,N) \leq \phi(d)$.
\end{theorem*}
More generally, they considered the case of several pairwise disjoint subsets and studied the joint distribution of the normalized counting variables and obtained a functional central limit theorem. 

In this paper we are concerned with two natural variations on this theme. Namely, we will consider the lattice point counting problem where the lattice is chosen at random from the spaces $(Y, \mu_Y)$ and $(Y_{{\vp}/q}, \mu_q)$.

We refer to these as the \emph{affine} lattice point counting problem and the \emph{congruence} lattice point counting problem respectively. We prove analogues of the results of Str\"{o}mbergsson and  S\"{o}dergren in the affine and congruence setting, and also analogues of results of Rogers \cite{Rogers56-2}, Schmidt \cite{Sch59} and S\"{o}dergren \cite{Sod} on Poissonian behaviour of lengths of lattice vectors in a randomly chosen lattice. See also related work of Kim \cite{Kim}. The main tool in \cite{StSo} is a version of Rogers's formula; in fact, one needs all moments, not just the second moment. In an analogous fashion, Theorems \ref{higher moment formula: affine} and \ref{higher moment formula} will play a starring role in the proofs of the results stated below.

\subsection{Counting Results}
Our first two results are analogues of S\"{o}dergren's results \cite{Sod} in the affine and congruence setting respectively. For each $d \geq 2$, let 
$\mathcal{S}=\mathcal S_d = \{S_t: t \geq 0\}$
be an increasing family of subsets of $\mathbb{R}^{d}$ with $\vol(S_t) = t$, and for $\Lambda \in  Y=\ASL_d(\ZZ)\setminus \ASL_d(\RR)$, set
\[
    N_t(\Lambda) := \#{\left(S_t \cap \Lambda\right)}.
\]
Denote by $\{N^\lambda(t) : t \geq 0\}$ a Poisson process on the non-negative real line with intensity $\lambda$.

\begin{theorem}[]\label{thmppmainaffine}
    The stochastic process $\{N_t(\Lambda) : t \geq 0\}$ converges weakly to $\{N^1(t) : t \geq 0\}$
as $d$ goes to infinity.
\end{theorem}

Let $q\in \NN_{\ge 2}$ be given. For each $d \geq 2$ consider $\mathcal{S} =\mathcal S_d= \{S_t: t > 0\}$, an increasing family of subsets of $\mathbb{R}^{d}$ and $\bm{p}/q \in \mathbb{Q}^{d}$ for some $\vp=\vp_d\in  \ZZ^d$ coprime with $q$.
By abuse of notation, set
\[
    N_t(\Lambda) = \#(S_t \cap \Lambda),
\]
for $\Lambda \in (Y_{\bm{p}/q}, \mu_q)$.

\begin{theorem}[]\label{thmppmaincong}
    \begin{enumerate}[(i)]
        \item For $q \geq 3$, the stochastic process $\{N_t(\Lambda) : t > 0\}$ converges weakly to $\{N^1(t) : t > 0\}$
as $d$ goes to infinity.
        \item For $q = 2$, assume that $S_t$'s are symmetric about origin, and let $\widetilde{N}_t = \frac{1}{2} N_t$.
            Then the stochastic process $\left\{\widetilde{N}_t(\Lambda) : t > 0\right\}$ converges weakly to $\{N^{1/2}(t) : t > 0\}$
as $d$ goes to infinity.
    \end{enumerate}
    
\end{theorem}

Next we establish a central limit theorem for the normalized error term in the lattice point problem for a random affine lattice. 
\begin{theorem}\label{Strom-Sod main thm: affine}
Let $\phi:\NN\rightarrow \RR_{>0}$ be a function for which 
\begin{equation}\label{condition for f}
\lim_{d\rightarrow \infty} \phi(d)=\infty
\quad\text{and}\quad 
\phi(d)=O_\varepsilon(e^{\varepsilon d}),\; \forall \varepsilon >0.
\end{equation} 
Consider a sequence $\{S_d\}_{d\in \NN}$ of Borel sets $S_d \subseteq \RR^d$ such that $\vol(S_d)=\phi(d)$.
Let
\[
Z^1_d=\frac {\#\left(\Lambda\cap S_d\right)- \phi(d)} {\sqrt {\phi(d)}},
\]
be the random variable with $\Lambda \in (Y, \mu_Y)$. 
Then
\[
Z^1_d \rightarrow \mathcal N(0,1)\text{ as } d\rightarrow \infty
\]
in distribution.
\end{theorem}
We now turn to the space $Y_{{\vp}/q}$ which can be viewed as a finite volume homogeneous space of $\SL_{d}(\RR)$ (see Section \ref{sec:cong}) and therefore inherits a natural finite Haar measure $\mu_q$.
\begin{theorem}\label{Strom-Sod main thm: congruence}
Let a function $\phi:\NN\rightarrow \RR_{>0}$ and a sequence $\{S_d\}$ of Borel sets be given as in Theorem~\ref{Strom-Sod main thm: affine}. When $q=2$, we further assume that each $S_d$ is symmetric with respect to the origin.
Let
\[
Z^{\vp/q}_d=\left\{\begin{array}{cl}
\dfrac {\#\left(\Lambda\cap S_d\right)- \phi(d)} {\sqrt {2\phi(d)}}, &\text{if } q=2;\\[0.2in]
\dfrac {\#\left(\Lambda\cap S_d\right)- \phi(d)} {\sqrt {\phi(d)}}, &\text{otherwise}
\end{array}\right.
\]
be a random variable associated with $\Lambda\in (Y_{\vp/q}, \mu_q)$. Then 
\[
Z^{\vp/q}_d \rightarrow \mathcal N(0,1)\text{ as } d\rightarrow \infty
\]
in distribution.
\end{theorem}

The next two theorems are functional central limit theorems in the affine and congruence case respectively.  
\begin{theorem}\label{Strom-Sod Brownian: affine} Let a function $\phi:\NN\rightarrow \RR_{>0}$ be given as in Theorem~\ref{Strom-Sod main thm: affine}.
Consider a sequence $\{S_d\}_{d\in \NN}$ of star-shaped Borel sets $S_d\subseteq \RR^d$ centered at the origin such that $\vol(S_d)=\phi(d)$.
Let us define the random function
\[
t\in [0,1]
\mapsto Z^1_d(t):= \frac {\#\left(\Lambda \cap t^{1/d}S_d\right) - t\phi(d)}{\sqrt{\phi(d)}},
\]
where $\Lambda$ is a random affine lattice in $(Y, \mu_Y)$.
Here, $tS=\{t\vv\in \RR^d : \vv\in S\}$ for any $t\in \RR_{\ge 0}$ and $S\subseteq \RR^d$.
Then $Z^1_d(t)$ converges in distribution to one-dimensional Brownian motion as $d$ goes to infinity. 
\end{theorem}

\begin{theorem}\label{Strom-Sod Brownian: congruence}
Let a function $\phi:\NN\rightarrow \RR_{>0}$ and a sequence $\{S_d\}_{d\in \NN}$ of Borel sets be as in Theorem~\ref{Strom-Sod  Brownian: affine}.
When $q=2$, we further assume that each $S_d$ is symmetric with respect to the origin.
Define the random function
\[
t\in [0,1]
\mapsto Z^{\vp/q}_d(t):= \left\{\begin{array}{cl}
\dfrac {\#\left(\Lambda \cap t^{1/d}S_d\right) - t\phi(d)}{\sqrt{2\phi(d)}}, &\text{if } q=2;\\[0.2in]
\dfrac {\#\left(\Lambda \cap t^{1/d}S_d\right) - t\phi(d)}{\sqrt{\phi(d)}}, &\text{otherwise.}\end{array}\right.
\]
Then $Z^{\vp/q}_d(t)$ converges in distribution to one-dimensional Brownian motion. 
\end{theorem}


\subsection*{Structure of the paper}
In section \ref{Section:Higher Moment Formulae}, we state and prove the moment formulae for the affine and congruence cases. In fact, we provide two approaches, one kindly suggested to us by the referee. Section 3 is devoted to the study of Poissonian behaviour. In particular, analogues of results of S\"{o}dergren \cite{Sod} and Rogers \cite{Rogers55-2, Rogers56} in the affine and congruence setting are established. These results might be of independent interest. Section 4 contains affine and congruence versions of the variation on Rogers' formula developed by Str\"{o}mbergsson and  S\"{o}dergren. Finally, Section 5 is devoted to the proofs of the counting results.

\subsection*{Acknowledgements}
We are very grateful to the anonymous referee for an extremely detailed report which pointed out several mistakes in an earlier version of the paper and also generously offered solutions to some of the issues. A. G. gratefully acknowledges support from a MATRICS grant from the Science and
Engineering Research Board, a grant from the Infosys foundation and a Department of
Science and Technology, Government of India, Swarnajayanti fellowship. 
J. H. was supported by a KIAS Individual Grant MG088401 at Korea Institute for Advanced Study.
The authors were supported by the Department of Atomic Energy, Government of India,
under project no.12-R\&D-TFR-5.01-0500.

\section{Higher Moment Formulae}\label{Section:Higher Moment Formulae}
We define
\[
    \ASL_d(\RR) := \left\{\left(\begin{array}{cc}
                g & 0 \\
                \xi & 1 \\
    \end{array}\right) : g\in \SL_d(\RR),\; \xi \in \RR^d \right\}
\]
and denote by $(\xi,g)$ an element of $\ASL_d(\RR)$.
One can identify the space of affine unimodular lattices with 
\[Y_d = Y =\ASL_d(\ZZ)\backslash \ASL_d(\RR)\]
via the map 
$$\ASL_d(\ZZ)(\xi,g) \mapsto \ZZ^d g+\xi.$$
We denote by $\mu_Y$ the Haar measure on $\ASL_d(\mathbb{R})$ normalized so that $$\mu_Y(\ASL_d(\mathbb{Z}) \backslash \ASL_d(\mathbb{R})) = 1.$$

Let $F:(\RR^d)^k \rightarrow \RR$ be a bounded function of compact support. Define the transform $\Siegel{k}{F}$ of $F$ by
\[
    \Siegel{k}{F}(\Lambda)=\sum_{\scriptsize \begin{array}{c}
            \vm_i
            \in \Lambda\\
    1\le i \le k\end{array}}
    F(\vm_1, \ldots, \vm_k),\; \Lambda\in \ASL_d(\ZZ)\backslash \ASL_d(\RR).
\]

 By a mild abuse of notation, we will use $\Siegel{k}{F}$ to also denote the function induced by the natural inclusion $$\SL_d(\ZZ)\backslash\SL_d(\RR)\hookrightarrow \ASL_d(\ZZ)\backslash \ASL_d(\RR).$$

\begin{notation}\label{notation}
We follow Rogers \cite{Rogers55} in setting some notation and recalling the definition of \emph{admissible} matrices. 
    \begin{enumerate}
        \item 
        We will identify the
        $k$-th power $(\RR^d)^k$
        of $\RR^d$ with $\Mat_{k,d}(\RR)$.
            For a matrix $D$, denote by $[D]^j$ the $j$-th column of $D$ and $[D]_i$ the $i$-th row of $D$.
        \item For $u\in \NN$ and $r\in \{1, \ldots, k\}$, the collection $\mathfrak D_{r,u}^k$ is the set of integral matrices $D=(d_{ij}) \in \Mat_{k, r}(\ZZ)$ such that the greatest common divisor of all elements of $D$ is one and there are $1\le i_1 < \ldots < i_r \le k$ with the following properties:
            \begin{enumerate}[(i)]
                \item ${\tp{([D]_{i_1}, \ldots, [D]_{i_r})}}=u\Id_r$;
                \item $d_{ij}=0$ for $1\le j \le r$ and $1\le i < i_j$.
            \end{enumerate}
            We say that $D$ is admissible if $D$ satisfies the above properties.
        \item For each $D\in \mathfrak D_{r,u}^k$, 
        \begin{enumerate}
        \item
        set $I_D :=\{i_1<\ldots<i_r\}$, where $i_1<\ldots<i_r$ are as  above;  
        \item let
        \[
        \Phi^{(d)}(D,u)=\left\{\left(\begin{array}{c} 
        \vn_1 \\
        \vdots \\
        \vn_r \end{array}\right) \in (\ZZ^d)^r : \frac D u \left(\begin{array}{c} 
        \vn_1 \\
        \vdots \\
        \vn_r \end{array}\right)\in (\ZZ^d)^k\quad\text{and}\quad \begin{array}{c}
        \vn_1, \ldots, \vn_r \text{ are}\\
        \text{linearly independent}\end{array}\right\};
        \]
        \item
        define $N(D,u)$ to be the number of vectors $\vv\in \{0, 1 \ldots, u-1\}^r$ for which $$\frac 1 u D \:{\tp{\vv}}\in \ZZ^k.$$
        \end{enumerate}
    \end{enumerate}
\end{notation}
We are now ready to state Rogers's famous integral formula for $\Siegel{k}{F}$ on $\SL_d(\ZZ)\backslash \SL_d(\RR)$ introduced in \cite{Rogers55}.

\begin{theorem}[Rogers \cite{Rogers55}] \label{Rogers moment}
    Let $F:(\RR^d)^k \rightarrow \mathbb{R}_{\geq 0}$,
where $1\le k \le d-1$,
be a bounded function of compact support. Then,
\[\begin{split}
\int_{X_d} \Siegel{k}{F}(\Lambda) \d\mu(\Lambda)
=
F\left(\begin{array}{c}
\origin\\
\vdots\\
\origin\end{array}\right)
+\sum_{r=1}^k \sum_{u\in \NN} \sum_{D\in \mathfrak D^k_{r,u}}
\frac {N(D,u)^d} {u^{dr}}
\int_{(\RR^d)^r} F \left( \frac D u \left(\begin{array}{c}
\vv_1 \\
\vdots \\
\vv_r \end{array}\right)\right)\d \vv_1 \cdots \d \vv_r.
\end{split}\]
\end{theorem}

        We note that Rogers did not comment on the nature of convergence of the RHS of the above equation. He did however mention \cite[second paragraph of page 279]{Rogers55} that results in another paper of his \cite[\S9]{Rogers55-2} imply absolute convergence for $d \geq [\tfrac{1}{4}k^2] + 2$).
        Schmidt \cite{Sch58} showed that in the case of a bounded compactly supported function $F : (\mathbb{R}^{d})^k \to \mathbb{R}_{\geq 0}$ the above sum is absolutely convergent, in other words both sides of the above equation are finite (and equal).
        Thus, Rogers's theorem holds also for a bounded compactly supported function $F : (\mathbb{R}^{d})^k \to \mathbb{R}$ and both sides of the above equation are finite in this case (since Rogers's theorem holds for $|F|$, we have absolute convergence of the sum and we can rearrange the terms in the sum).

Theorem~\ref{Rogers moment} follows from the fact that
\[
(\ZZ^d)^k=
\left\{\tp{(\origin,\ldots, \origin)}\right\} \sqcup  \bigsqcup_{r=1}^k \bigsqcup_{u\in \NN} \bigsqcup_{D\in \mathfrak D^k_{r,u}} \frac D u \Phi^{(d)}(D,u)
\]
and the following proposition.

\begin{proposition}[Rogers \cite{Rogers55}]\label{Rogers component wise}
Let $F:(\RR^d)^k \rightarrow \RR$ be a bounded function of compact support.
For each $D\in \mathfrak D^k_{r,u}$,
we have
\[
\int_{X_d}
\sum_{\scriptsize \begin{array}{c}
\tp{(\vn_1, \ldots, \vn_r)}\\
\in \Phi^{(d)}(D,u)\end{array}}
F\left(\frac D u \left(\begin{array}{c}
\vn_1 \\ \vdots \\ \vn_r\end{array}\right) g\right) \d\mu(g)
=
\frac {N(D,u)^d} {u^{dr}}
 \int_{(\RR^d)^r} F \left( \frac D u \left(\begin{array}{c}
\vv_1 \\
\vdots \\
\vv_r \end{array}\right)\right)\d \vv_1 \cdots \d \vv_r.
\]
\end{proposition}

\subsection{Higher Moment formulae for $Y$}
In \cite{EMV2015}, El-Baz, Marklof and Vinogradov established a second moment formula for the Siegel transform on $ Y=\ASL_2(\ZZ)\backslash \ASL_2(\RR)$ which easily extends to the case when $d\ge 3$ (see \cite[Appendix B]{EMV2015}).
We will generalize their result to higher moment formulae for the transform $\Siegel{k}{\cdot}$ on $Y$. It is well-known that 
\[
    \bigcup_{g\in \mathcal F}
    \left\{(\xi,g) : \xi \in [0,1)^d g \right\}
\]
is a fundamental domain for $Y$, where $\mathcal F$ is any fixed fundamental domain for $\SL_d(\ZZ)\backslash \SL_d(\RR)$.
Thus one can take the probability $\ASL_d(\RR)$-invariant measure $\mu_Y$ on $Y$ as the measure inherited from the product of the Haar measure $\mu$ on $\SL_d(\RR)$ and the Lebesgue measure on $\RR^d$.

\begin{theorem}\label{higher moment formula: affine 1}
    Let $F:(\RR^d)^k\rightarrow \RR$ be a bounded compactly supported function, and $d\ge 2$. We have the following:
    \begin{enumerate}[(i)]
        \item   For $k = 1$,
        \begin{equation}\label{eq 1: higher moment formula: affine 1 for k=1}
        \begin{split}
            &\int_Y
             \Siegel{1}{F}(\Lambda) \d\mu_Y(\Lambda)
             = \int_{\RR^d} F(\vy) \d\vy
        \end{split}
        \end{equation}
        \item   For $2\le k\le d$,
        \begin{equation}\label{eq 1: higher moment formula: affine 1}
        \begin{split}
            &\int_Y
             \Siegel{k}{F}(\Lambda) \d\mu_Y(\Lambda)
             = \int_{(\RR^d)^k} F\left(\begin{array}{c}
                     \vy_1\\
                     \vy_2\\
                     \vdots\\
             \vy_k\end{array}\right) \d\vy_1\d\vy_2\cdots \d\vy_k
             +\int_{\RR^d} F\left(\begin{array}{c}
                     \vy_1\\
                     \vy_1\\
                     \vdots\\ 
             \vy_1\end{array}\right) \d\vy_1
             \\
              &\hspace{0.4in}+\sum_{r=1}^{k-2} \sum_{u\in \NN} \sum_{D\in \mathfrak D^{k-1}_{r,u}} \frac {N(D,u)^d} {u^{dr}}\int_{(\RR^d)^{r+1}}
             F\left(D' \left(\begin{array}{c}
                         \vy_1\\
                         \vy_2\\
                         \vdots\\
             \vy_{r+1}\end{array}\right)\right)\d\vy_1\d\vy_2\cdots \d\vy_{r+1},
         \end{split}
     \end{equation}
    where $D'$ for $D\in \mathfrak D^{k-1}_{r,u}$ is $k\times (r+1)$ matrix defined by
    \begin{equation}\label{D':affine}
    D'=\left(\begin{array}{c|c}
                1 & 0\;\cdots\;0 \\ \hline
                \begin{array}{c}
                    1 \\
                    \vdots\\
            1\end{array} & \dfrac D u \end{array}\right).
            \end{equation}
    Here as a convention, for $k=2$, let us assume that $\sum_{r=1}^{0}$ is the empty summation.
    \end{enumerate}

    Finally, both sides of the equation \eqref{eq 1: higher moment formula: affine 1} are finite.
\end{theorem}

\begin{proof}

We first remark that the $k=1$ case is classical and can be proved using the \emph{folding and unfolding} argument. When $k=2$, the result can be  deduced from \cite[Appendix B]{EMV2015}, where the authors proved the second moment formula for $d=2$. However, their proof can be seen to work in full generality. We will therefore focus on the case when $k\ge 3$.

    Fix any fundamental domain $\mathcal F$ for $\SL_d(\ZZ)\backslash \SL_d(\RR)$. For each $g\in \mathcal F$, by the change of variables $\xi=\eta g$, we have
    \[\begin{split}
        \int_Y
        \Siegel{k}{F}(\ZZ^d g+\xi) \d\mu(g)\d\xi
&=\int_Y
\Siegel{k}{F}((\ZZ^d +\eta)g) \d\mu(g)\d\eta\\
 &=\int_{\mathcal F}\int_{[0,1)^d}\sum_{\scriptsize \begin{array}{c}
         \vm_i\in \ZZ^d\\
 1\le i\le k\end{array}}
 F\left(\begin{array}{c}
         (\vm_1 +\eta)g\\ 
         (\vm_2+\eta)g\\ 
         \vdots\\
 (\vm_k +\eta)g\end{array}\right)\d\eta \d\mu(g).
\end{split}\]

For each $g\in \mathcal F$ and $\vm_1\in \ZZ^d$, put $\vy_1=(\eta+\vm_1)g$ and $\vm'_j=\vm_j-\vm_1$ for $2\le j\le k$.
Since $\bigcup_{\vm_1\in \ZZ^d}\left(\vm_1+[0,1)^d\right)= \RR^d$, the above expression is
\[\begin{split}
&=\int_{\mathcal F}
\int_{\RR^d}
\sum_{\scriptsize \begin{array}{c}
        \vm'_j\in \ZZ^d\\
2\le j\le k\end{array}}
F\left(\begin{array}{c}
        \vy_1\\
        \vy_1+\vm'_2 g\\
        \vdots\\
\vy_1+\vm'_k g\end{array}\right)\d\vy_1 \d\mu(g)\\
&=\int_{\RR^d} F\left(\begin{array}{c}
        \vy_1\\
        \vy_1\\
        \vdots\\ 
\vy_1\end{array}\right) \d\vy_1
+\int_{(\RR^d)^k} F\left(\begin{array}{c}
        \vy_1\\
        \vy_2\\
        \vdots\\
\vy_k\end{array}\right) \d\vy_1\d\vy_2\cdots \d\vy_k\\
&+\sum_{r=1}^{k-2} \sum_{u=1}^\infty \sum_{D\in \mathfrak D^{k-1}_{r,u}} \frac {N(D,u)^d} {u^{dr}}\int_{(\RR^d)^{r+1}}
F\left(D' \left(\begin{array}{c}
            \vy_1\\
            \vy_2\\
            \vdots\\
\vy_{r+1}\end{array}\right)\right)\d\vy_1\d\vy_2\cdots \d\vy_{r+1},
\end{split}\]
where $D'$ is defined as in \eqref{D':affine}
In the last equality, we applied Theorem~\ref{Rogers moment} to the function
    \[
        F' : (\vy_2, \ldots, \vy_k) \mapsto \int_{\RR^d} F\left(\vy_1, \vy_1+\vy_2, \ldots, \vy_1+\vy_k\right) \d\vy_1.
    \]

    Observe that it is enough to prove finiteness for $F \geq 0$.
    Indeed, for general $F$ finiteness for $|F|$ proves the absolute convergence of the sum in the RHS of \eqref{eq 1: higher moment formula: affine 1}.
    We note that (for $F \geq 0$) $F'$ is a compactly supported bounded positive function and hence invoking Schmidt \cite[Theorem 2]{Sch58} for this function proves our claim.

\end{proof}

\subsection{Higher Moment Formulae for $Y_{\vp/q}$}\label{sec:cong}

Recall that for $\vp\in \ZZ^d \smallsetminus \{\origin\}$ and $q\in \NN_{\ge 2}$ such that $\gcd(\vp, q)=1$, we set
\[
    Y_{{\vp}/q}:=\left\{\left(\ZZ^d+\frac {\vp} q\right)g : g\in \SL_d(\RR) \right\}\subseteq Y.
\]
We remark that the space $Y_{\bm{p}/q}$ doesn't depend on $\bm{p}$ because $Y_{\bm{p}/q}$ is the space of all affine grids $L + \bm{v}$, where $L$ is an unimodular lattice in $\mathbb{R}^{d}$ and $\bm{v} \in \mathbb{R}^{d}$ is a representative torsion point of order $q$ in the torus $\mathbb{R}^{d}/L$. Indeed, one can see that for such $L + \bm{v}$, $\exists~ g \in \SL_{d}(\mathbb{R})$ such that $L = \mathbb{Z}^{d}g$, and since $q\bm{v} \in L$ we have $\bm{v} = \frac{\bm{w}g}{q}$, where $\bm{w} \in \mathbb{Z}^{d}$ and $\bm{w}$ is of order $q$ in $(\mathbb{Z}/q\mathbb{Z})^d$ (since $\bm{v}$ is of order $q$).
Therefore, $L + \bm{v} = {\left(\mathbb{Z}^{d} + \frac{\bm{w}}{q}\right)}g$.
Since $\bm{p}$ is also of order $q$ in $(\mathbb{Z}/q\mathbb{Z})^d$ and $\SL_{d}(\mathbb{Z})$ acts transitively on elements of order $q$ in $(\mathbb{Z}/q\mathbb{Z})^d$, $\exists~ \gamma \in \SL_{d}(\mathbb{Z})$ such that $\bm{w} = \bm{p}\gamma$.
Hence
\[
    L + \bm{v} = {\left(\mathbb{Z}^{d} + \frac{\bm{p}\gamma}{q}\right)}g = {\left(\mathbb{Z}^{d} + \frac{\bm{p}}{q}\right)} \gamma g.
\]

Let $\{\ve_j\}$ be the canonical basis of $\RR^d$.
Define
\[\begin{split}
    \Gamma(q)&=\left\{\gamma\in \SL_d(\ZZ): \gamma \equiv \Id_d\mod q\right\},\\
    \Gamma_1(q)&=\left\{\gamma\in \SL_d(\ZZ): \ve_1\gamma\equiv \ve_1\mod q\right\},
\end{split}\]
and $X_q=\Gamma(q)\backslash \SL_d(\RR)$.
If we choose any $\gamma_{{\vp}}\in \SL_d(\ZZ)$ for which ${\vp}=r\ve_1 \gamma_\vp$, where $r=\gcd {\vp}$, then $Y_{{\vp}/q}$ can be identified with $\gamma_{{\vp}}^{-1} \Gamma_1(q)\gamma_{\vp}\backslash \SL_d(\RR)$ (\cite[Lemma 3.1]{GKY2020}).
Denote by $\mu_q$ the Haar measure on $\SL_d(\mathbb{R})$ normalized so that $\mu_q(Y_{\bm{p}/q}) = 1$. More precisely, let $J_q=[\SL_d(\ZZ):\Gamma_1(q)]$. We can see that $\mu_q=\frac 1 {J_q} \mu$, which is independent of the choice of $\vp$.

Recall that we identify the $k$-tuple $(\RR^d)^k$ of $\RR^d$ with $\Mat_{k,d}(\RR)$. Let $\{E_{ij} : 1\le i \le k, 1\le j\le d\}$ be the standard basis for $(\RR^d)^k$, that is, the $(k,\ell)$-entry $[E_{ij}]_{k\ell}=0$ except that $[E_{ij}]_{ij}=1$.


The Lemma below essentially follows from the definition. However we provide a proof since it is vital in setting up and proving moment formulas for congruence quotients.

\begin{lemma}\label{lemma:(D/u)Lambda_D}
For each $D\in \mathfrak D^k_{r,u}$, where $\mathfrak D^k_{r,u}$ is as in Notation~\ref{notation}, define
\[
\Lambda_{D}=\left\{\left(\begin{array}{c}
\ell_1\\ 
\vdots\\
\ell_r\end{array}\right)\in \ZZ^r : \frac D u \left(\begin{array}{c}
\ell_1\\ 
\vdots\\
\ell_r\end{array}\right)\in \ZZ^k \right\}.
\]

It follows that $\frac D u: \Lambda_D \rightarrow \frac D u \RR^r$ is injective and moreover,
\[
\frac D u \Lambda_D= \frac D u\RR^r \cap \ZZ^k.
\]
In other words, the set $\frac D u \Lambda_D$ is a primitive sublattice of $\ZZ^k$ of rank $r$, which is given by intersecting with the rational subspace $\frac D u\RR^r\subseteq \RR^k$.
\end{lemma}
\begin{proof}
    One direction as well as the injectivity is obvious. Let us show the other direction. Suppose that $\vl\in \RR^r$ satisfies that $\frac D u\vl\in \ZZ^k$. Considering indices $1\le i_1< \ldots < i_r\le k$ in Notation~\ref{notation} (2), we have that $\vl=([\frac D u \vl]^{i_1}, \ldots, [\frac D u \vl]^{i_r})\in \ZZ^r$. This proves the lemma since $\Lambda_D=\ZZ^r \cap \left(\frac D u\right)^{-1}\ZZ^k$.
\end{proof}

\begin{notation}\label{notation:cong 1}
For each $D\in \mathfrak D^k_{r,u}$, since $\Lambda_D$ defined as in Lemma~\ref{lemma:(D/u)Lambda_D} is primitive, one can find elements $\vb_{1}, \ldots, \vb_{k-r}$ in $\ZZ^k$ such that for any $\ZZ$-basis $\{\vb_{k-r+1}, \ldots, \vb_{k}\}$ of $\frac D u \Lambda$, it holds that 
\[
\ZZ^k=\ZZ\vb_1 \oplus \cdots \oplus \ZZ\vb_k.
\]

Fix such a set $\{\vb_{1}, \ldots, \vb_{k-r}\}$ for each $D\in \mathfrak D^k_{r,u}$ and denote 
\[\mathcal R(D)= \ZZ\vb_1 \oplus \cdots \oplus \ZZ\vb_{k-r}\] 
so that $\ZZ^k=\bigsqcup_{\vl\in \mathcal R(D)} \left(\vl+\frac D u \Lambda_D\right)$.
We also define the set $P_t(\mathcal R(D))$ for every $t\in \NN$ with $\gcd(t,q)=1$ as
\[
P_t(\mathcal R(D))=\{\vl\in \mathcal R(D): \gcd(\vl, t)=1\}.
\]
\end{notation}

We are now ready to formulate the higher moment formula for $Y_{\vp/q}$, based on Notation~\ref{notation:cong 1}. The formula in equation  \eqref{higher moment formula: eq (1)1} below depends on a choice of $\mathcal R(D)$ for each $\mathfrak D^k_{r,u}$. We are very grateful to the anonymous referee for providing an alternative formulation which does not involve any ad-hoc choices. This formulation can be found in Theorem \ref{higher moment formula}.  We have chosen to include both formulations because we believe that \eqref{higher moment formula: eq (1)1} is more `intrinsic' in some sense, i.e. more indicative of the proof, 
see for instance the similarity with the second moment formula proven in \cite{GKY2020} (see also \cite[Proposition 7.6]{MS2010}). 
\begin{theorem}\label{higher moment formula 1}
    Let $d\ge 3$ and $1\le k \le d-1$.
    Let $F:(\RR^d)^k\rightarrow \RR$ be bounded and compactly supported. Then
    \begin{enumerate}
        \item For $k = 1$,
        \[
        \int_{Y_{{\vp}/q}} \Siegel{1}{F} (\Lambda) \d\mu_q(\Lambda)
=\int_{\RR^d} F\left({\vy}\right) \d\vy
\]
        \item For $2 \leq k \leq d-1$,
            \begin{equation}\label{higher moment formula: eq (1)1}
                \begin{split}
    &\int_{Y_{{\vp}/q}} \Siegel{k}{F} (\Lambda) \d\mu_q(\Lambda)
    =\int_{(\RR^d)^k} F\left({\tp{(\vy_1, \ldots, \vy_k)}}\right) \d\vy_1 \cdots \d\vy_k\\
    &+\int_{\RR^d} F\left({\tp{(\vy, \ldots, \vy)}}\right) d\vy
    +\hspace{-0.1in}\sum_{\scriptsize \begin{array}{c} t\in \NN \\ (t,q)=1 \end{array}}
    \hspace{-0.1in}\sum_{\scriptsize \begin{array}{c} \vl\neq \origin \\ \in \ZZ^{k-1}\end{array}}
    \int_{\RR^d}F\left(\left(\begin{array}{c}
                t\vy \\
                (t+\ell_1q)\vy \\
    \vdots \\ (t+\ell_{k-1}q)\vy\end{array}\right)\right)d\vy\\
    &+\sum_{r=1}^{k-2} \sum_{u\in \NN} \sum_{D\in \mathfrak D^{k-1}_{r,u}}
     \left[\frac {N(D,u)^d}{u^{dr}} \int_{(\RR^d)^{r+1}} F\left(D'\left(\begin{array}{c}
    \vy_1 \\ \vdots \\ \vy_{r+1}\end{array}\right)\right) \d\vy_1 \cdots \d\vy_{r+1}\right.\\
    &\hspace{0.5in}\left. \sum_{\scriptsize \begin{array}{c} t\in \NN \\ (t,q)=1 \end{array}} \sum_{\scriptsize \begin{array}{c} \vl\in \\ P_t(\mathcal R(D))\end{array}}
        \frac{N(D,u)^d}{t^d\cdot u^{dr}}
        \int_{(\RR)^{r+1}} F\left(D'_{t,\vl}  \left(\begin{array}{c}
    \vy_1 \\ \vdots \\ \vy_{r+1}\end{array}\right)\right)d\vy_1 \cdots d\vy_{r+1}\right],\\
                \end{split}
            \end{equation}
where $D'$ and $D'_{t,\vl}$ for $D\in \mathfrak D^{k-1}_{r,u}$ and $\vl={\tp{(\ell_1, \ldots, \ell_{k-1})}}\in P_t(\mathcal R(D))$ are $k\times (r+1)$ matrices defined as follows:
\[
D'=\left(\begin{array}{c|c}
1 & 0 \, \cdots \, 0 \\
\hline
\begin{array}{c}
1 \\ \vdots \\ 1 \end{array} & \dfrac 1 u D 
\end{array}\right)
\quad\text{and}\quad
D'_{t,\vl}
=\left(\begin{array}{c|c}
t & 0 \, \cdots \, 0 \\
\hline
\begin{array}{c}
t+\ell_1q \\ \vdots \\ t+\ell_{k-1}q \end{array} & \dfrac 1 u D 
\end{array}\right)
\]
Here, if $k=2$, we will consider $\sum_{m=1}^0$ as the empty summation.
    \end{enumerate}

    Finally, both sides of the equation \eqref{higher moment formula: eq (1)1} are finite.
\end{theorem}

Notice that the right hand side of the above expression does not depend on $\vp\in \ZZ^d \smallsetminus \{\origin\}$, once $\gcd(\vp,q)=1$.\\

We need several lemmas for the proof of Theorem \ref{higher moment formula 1}. 
Let
                \[
                    H=\left\{\left(\begin{array}{cc}
                                1 & 0 \\
                    \tp{\vv'} & g' \end{array}\right) : \vv'\in \RR^{d-1}\;\text{and}\; g'\in \SL_{d-1}(\RR) \right\}
                \]
                and denote an element of $H$ by $[\vv', g']$.
                Let us identify $\SL_{d-1}(\RR)$ with the subgroup $\{[0, g']:g'\in \SL_{d-1}(\RR)\}$ of $H$. One can define the Haar measure $\mu_H$ on $H$ by the product of $\mu'$ and the Lebesgue measure on $\RR^{d-1}$, where $\mu'$ is the Haar measure such that $\mu'(X_{d-1})=1$.

                Notice the difference between $H$ and $\ASL_{d-1}(\RR)$. For instance, a fundamental domain of $(\SL_d(\ZZ)\cap H)\backslash H$ is given by $[0,1)^{d-1} \times \mathcal F_{d-1}$, where $\mathcal F_{d-1}$ is a fundamental domain of $\SL_{d-1}(\ZZ)\backslash \SL_{d-1}(\RR)$,
                whereas that of $\ASL_{d-1}(\ZZ)\backslash \ASL_{d-1}(\RR)$ is given by 
                \[\left\{[\xi'g', g']: g'\in \mathcal F_{d-1}\;\text{and}\;\xi' \in [0,1)^{d-1}\right\}.\]
                
                \begin{proposition}\label{MS Lemma 7.7}
                    Let $F:(\RR^d)^k \rightarrow \RR_{\ge 0}$, where $d\ge 3$ and $1\le k \le d-2$, be a bounded and compactly supported function.
                    Suppose that $\xi=(z_1, \xi')\in \RR^d$ with $z_1\in \RR$ and $\xi'\in \ZZ^{d-1}$. Then,
                    \[\begin{split}
&\int_{\SL_d(\ZZ)\cap H\backslash H}
\Siegel{k}{F}\left((\ZZ^d+\xi)g\right)\d\mu_H(g)\\
&=\sum_{\ell_1, \ldots, \ell_k\in \ZZ} F\left(\sum_{i=1}^k (z_1+\ell_i)E_{i1}\right)\\
&\hspace{0.2in}+\sum_{r=1}^k \sum_{u\in\NN} \sum_{D\in \mathfrak D^k_{r,u}}
\sum_{\scriptsize \begin{array}{c}
        {\tp{(\ell_1, \ldots, \ell_k)}} \\
\in \mathcal R(D)\end{array}}\\
&\hspace{0.5in}\frac {N(D,u)^{d}} {u^{dr}}\int_{(\RR^d)^r}
F\left(\sum_{i=1}^k (z_1+\ell_i)E_{i1}+
    \frac {D} {u} \left(\begin{array}{c}
            \vx_1 \\
            \vdots\\
\vx_r \end{array}\right)\right)\d\vx_1 \cdots \d\vx_r.
                    \end{split}\]
                \end{proposition}

Note that $H$ is the isotropy subgroup of $\ve_1$ in $\SL_d(\RR)$.
We will compute the integral $\int_{\SL_d(\ZZ)\cap H \backslash H} \Siegel{k}{F} \d\mu_H$ in two steps: we first process the integrals associated to the first column in $(\RR^d)^k\simeq \Mat_{k,d}(\RR)$, and then apply Theorem~\ref{Rogers moment} to the integrals associated to the remaining columns.
For this, we need the lemma below which describes the relation between the primitive sublattice $\frac D u \Lambda_D$ of $\ZZ^k$ for $D\in \mathfrak D^k_{r,u}$ and its sublattice $\frac C w \Lambda_C$ for some $C\in \mathfrak D^k_{r-1,w}$.

\begin{lemma}\label{Lemma of MS Lemaa 7.7}
Recall Notation~\ref{notation:cong 1}. Let $D\in \mathfrak D^k_{r,u}$ with $r\ge 2$.
\begin{enumerate}[(a)]
\item For $1\le j_0\le r$ and $a_1, \ldots, a_{j_0-1}\in \QQ$, define $C_1\in \Mat_{k,r-1}(\QQ)$ by
\[
\left[C_1\right]^j
=\left\{\begin{array}{cl}
[D/u]^j +a_j[D/u]^{j_0} & \text{for } 1\le j<j_0; \\[0.05in]
{[D/u]^{j+1}} & \text{for } j_0\le j \le r-1.
\end{array}\right.
\] 
Let $w\in \NN$ be the least common denominator of $C_1$ and $C:=wC_1\in \mathfrak D^k_{r-1,w}$.
Let $\mathfrak D^{k}_{r-1,w}(D)$ be the collection of such matrices $C$. There is a one-to-one correspondence between 
\[
\bigcup_{w\in \NN}\mathfrak D^k_{r-1,w}(D)
\;\text{and}\;\left\{\text{$(r-1)$-dimensional rational subspaces in $D\RR^r$}\right\}.
\]

\item
For each $C\in \mathfrak D^k_{r-1,w}(D)$, define
\[
\Lambda_D(C)=\left\{\tp{(\ell_1, \ldots, \ell_r)} \in \Lambda_D : a_1 \ell_1 + \cdots + a_{j_0-1}\ell_{j_0-1}=\ell_{j_0} \right\}.
\]

Then there is a natural isomorphism from $\Lambda_C$ to $\Lambda_D(C)$ so that 
$$\frac D u \Lambda_D(C)= \frac C w \Lambda_C \subseteq \frac D u \Lambda_D \subseteq \ZZ^k.$$ 

For each such pair $(D,C)$, 
one can choose (and fix from now on)
an element $\vb_D\in \Lambda_D-\Lambda_D(C)$ so that if we let $\mathcal R_D(C)=\ZZ \vb_D \smallsetminus \{\origin\}$, then it holds that
\[
\Lambda_D-\Lambda_D(C)=\bigsqcup_{\vl\in \mathcal R_D(C)} \left(\vl + \Lambda_D(C)\right).
\]

\item For a given $C\in \mathfrak D^k_{r-1,w}$, let $D_1\in \mathfrak D^k_{r,u_1}$ and $D_2\in \mathfrak D^k_{r,u_2}$ be such that $D_1\neq D_2$ and $C \in \mathfrak D^k_{r-1,w}(D_1) \cap \mathfrak D^k_{r-1,w}(D_2)$. Then
\[
\frac {D_1} {u_1} \mathcal R_{D_1}(C)\cap \frac {D_2} {u_2} \mathcal R_{D_2}(C)=\emptyset.
\]
Hence, for any $\vl_1 \in \frac {D_1} {u_1} \mathcal R_{D_1}(C)$ and $\vl_2\in \frac {D_2} {u_2} \mathcal R_{D_2}(C)$, it follows that \[
\left(\vl_1 + \frac C w \Lambda_C\right) \cap \left(\vl_2 + \frac C w \Lambda_C\right) = \emptyset.
\]
\item For a given $C\in \mathfrak D^k_{r-1,w}$, one can choose $\mathcal R(C)$ in Notation~\ref{notation:cong 1} to be

\[\mathcal R(C)=\left\{{\tp{(0,\ldots, 0)}}\right\} \sqcup \bigsqcup_{u\in \NN} \bigsqcup_{\scriptsize \begin{array}{c} 
D\in \mathfrak D^k_{r,u}\;\text{s.t.}\\
C\in \mathfrak D^k_{r-1,w}(D)\end{array}}\frac D u \mathcal R_D(C)
\]
and vice versa.
\end{enumerate}
\end{lemma}
\begin{proof}
(a) One way is obvious from its construction. Suppose that $W\subseteq D\RR^r$ is a codimension-one rational subspace of $D\RR^r$. Then there is $C\in \mathfrak D^{k}_{r-1,w}$ so that $\frac C w \RR^{r-1}=W$. We want to show that $C\in \mathfrak D^k_{r-1,w}(D)$. Pick any $\tp{(\vm_1, \ldots, \vm_k)}\in (\RR^d)^k\in \frac C w (\RR^d)^k \cap (\ZZ^d)^k$ with rank $(r-1)$ and $\tp{(\vm'_1, \ldots, \vm'_r)}$ and $\tp{(\vm''_1, \ldots, \vm''_{r-1})}$ be such that 
\[
\frac D u \left(\begin{array}{c}
\vm'_1\\
\vdots \\
\vm'_r\end{array}\right)= \left(\begin{array}{c}
\vm_1\\
\vdots \\
\vm_k \end{array}\right)= \frac C w \left(\begin{array}{c}
\vm''_1 \\
\vdots \\
\vm''_{r-1}\end{array}\right).
\]

Let $I_D=\{i_1<\ldots<i_r\}$ be as in Notation~\ref{notation} (3). Since $\rk{\tp{(\vm_1, \ldots, \vm_k)}}=r-1$, and by definition of $D$ and $C$, there is $1\le j_0\le r$ for which
\[\begin{gathered}
\vm'_1=\vm_{i_1}=\vm''_1, \;\ldots,\;\vm'_{j_0-1}=\vm_{i_{j_0-1}}=\vm''_{j_0-1};\\
\vm'_{j_0}=a_1\vm'_1+\cdots+a_{j_0-1}\vm'_{j_0-1}\;\text{for some}\; a_1, \ldots, a_{j_0-1}\in \QQ;\\
\vm'_{j_0+1}=\vm_{i_{j_0+1}}=\vm''_{j_0}, \;\ldots,\; \vm'_{r}=\vm_{i_r}=\vm''_{r-1}.
\end{gathered}\]
It is easily seen that $C_1$ constructed from $D$ with $j_0$ and $a_1, \ldots, a_{j_0-1}\in \QQ$ in Property (a) is equal to $C$.

\vspace{0.1in}
\noindent (b) It is obvious from the definition that $C\in \mathfrak D^k_{r-1,w}$. The map 
\[\begin{split}
{\tp{(\ell_1, \ldots, \ell_{r-1})}} \mapsto
{\tp{(\ell_1, \ldots, \ell_{j_0-1}, a_1 \ell_1 + \cdots + a_{j_0-1}\ell_{j_0-1}, \ell_{j_0}, \ldots, \ell_{r-1})}}
\end{split}\]
gives an isomorphism from $\Lambda_C$ to $\Lambda_D(C)$ and by definition, 
\[\frac C w {\tp{(\ell_1, \ldots, \ell_{r-1})}}= \frac D u {\tp{(\ell_1, \ldots, \ell_{j_0-1}, a_1 \ell_1 + \cdots + a_{j_0-1}\ell_{j_0-1}, \ell_{j_0}, \ldots, \ell_{r-1})}} \in \ZZ^k.
\]

Recall Lemma \ref{lemma:(D/u)Lambda_D}. Since $\frac C w \Lambda_C$ is primitive, there is an element $\vb\in \frac D u \Lambda_D$ for which $\frac D u \Lambda_D=\frac C w \Lambda_C\oplus \ZZ\vb$. Set $\vb_D:=\left(\frac D u\right)^{-1}\vb$.

\vspace{0.1in}
\noindent (c) 
Let $\mathcal R_{D_i}(C)$ be generated by $\vb_{D_i}$ for $i=1,2$.
From the fact that $\frac {D_1} {u_1} \RR^r \cap \frac {D_2} {u_2} \RR^r = \frac {C} {w} \RR^{r-1}$, in other words,
\[
\frac {D_i}{u_i}\RR^r=\frac C w \RR^{r-1}\oplus \RR \left(\frac {D_i}{u_i}\vb_{D_i}\right)\;(i=1,2),
\]
it is obvious that $\frac {D_1}{u_1}\mathcal R_{D_1}(C) \cap \frac {D_2}{u_2}\mathcal R_{D_2}(C)=\emptyset$.
Moreover, for any $\vl_i\in \frac {D_i}{u_i}$, where $i=1,2$, 
\[\vl_i+\frac C w \Lambda_C\subseteq \vl_i+\frac C w \RR^{r-1}
\]
which are affine subspaces of $\frac C w \RR^{r-1}$ lying on $\frac {D_i} {u_i} \RR^r-\frac C w \RR^{r-1}$ for $i=1,2$, respectively, hence they are disjoint.


\vspace{0.1in}
\noindent To deduce (d) from (c), it suffices to show that
\[
\frac C w \Lambda_C + \Bigg(\left\{{\tp{(0,\ldots, 0)}}\right\} \sqcup \bigsqcup_{u\in \NN} \bigsqcup_{\scriptsize \begin{array}{c} 
D\in \mathfrak D^k_{r,u}\\
C\in \mathfrak D^k_{r-1,w}(D)\end{array}}\frac D u \mathcal R_D(C)\Bigg)
=\frac C w \Lambda_C + \mathcal R(C).
\]

Let $\vl \in \mathcal R(C)$ be given. Since $\frac C w \RR^{r-1} \oplus \RR\vl$ is a rational subspace of rank $r$, there are $u\in \NN$ and $D\in \mathfrak D^k_{r,u}$, which are uniquely determined, such that $\frac C w \RR^{r-1} \oplus \RR \vl=\frac D u \RR^r$.
It is obvious that $\vl \in \frac D u \Lambda_D - \Lambda_D(C)$, hence there is $\vl'\in \frac D u\mathcal R_D(C)$ so that
$$\vl\in \vl'+\frac D u\Lambda_D(C)=\vl'+ \frac C w \Lambda_C,$$
as asserted in the claim.
\end{proof}     
                
\begin{proof}[Proof of Proposition~\ref{MS Lemma 7.7}]
Fix a fundamental domain $\mathcal F'$ of $\SL_{d-1}(\ZZ)\backslash \SL_{d-1}(\RR)$ so that $\mathcal F'\times [0,1)^{d-1}$ is a fundamental domain of $(\SL_d(\ZZ)\cap H)\backslash H$.

Recall that $(\ZZ^d)^k \smallsetminus \{{\tp{(\origin, \ldots, \origin)}}\}$ is partitioned into
                    $
                        \bigsqcup_{r=1}^k
                        \bigsqcup_{u\in\NN}
                        \bigsqcup_{D\in \mathfrak D^k_{r,u}}
                        \frac {D} u \Phi^{(d)}(D,u),
                    $
                    where $\mathfrak D^k_{r,u}$ and $\Phi^{(d)}(D,u)$ are as in Notation~\ref{notation}.

By taking $g=[\vv',g']$ 
and from Rogers' formula, we have that 
                    \[\begin{split}
&\int_{(\SL_d(\ZZ)\cap H)\backslash H}
\Siegel{k}{F} \left((\ZZ^d+\xi)[\vv',g']\right)\d\mu_{H}([\vv',g']) \\
&=\int_{\mathcal F'\times [0,1)^{d-1}}\hspace{-0.18in}
\sum_{\scriptsize \begin{array}{c}
        \ell_i\in \ZZ\\
        \vm'_i\in \ZZ^{d-1}\\
        1\le i \le k
\end{array}}\hspace{-0.18in}
F\left(\begin{array}{c|c}
(\ell_1+z_1)+\vm'_1{\tp{\vv'}} & \vm'_1g' \\
\vdots & \vdots \\
(\ell_k+z_1)+\vm'_k{\tp{\vv'}} & \vm'_kg'
\end{array}
\right)\d\vv'\d\mu'(g')\\
&= F\left(\begin{array}{c|c}
z_1 & 0, \ldots, 0 \\
\vdots & \vdots \\
z_1 & 0, \ldots, 0\end{array}\right)
+\sum_{r=1}^k \sum_{u\in \NN} \sum_{D\in \mathfrak D^k_{r,u}}\int_{\mathcal F'\times [0,1)^{d-1}} \sum_{\scriptsize \begin{array}{c}
        {\tp{(\vn_1, \ldots, \vn_r)}}\\
\in \Phi^{(d)}(D,u)\end{array}}\\
&\hspace{1in} F\left(\begin{array}{c|c}
\left(\begin{array}{c}
z_1 \\ \vdots \\ z_1\end{array}\right)+ \dfrac D u \left(\begin{array}{c} \ell'_1 \\ \vdots \\ \ell'_r\end{array}\right)+ \dfrac D u \left(\begin{array}{c}
\vn'_1 \\ \vdots \\ \vn'_r\end{array}\right){\tp{\vv'}}
& \dfrac D u \left(\begin{array}{c}
\vn'_1 \\ \vdots \\ \vn'_r\end{array}\right)g'
\end{array}\right) \d\vv'\d\mu'(g),
                    \end{split}\]
where $\vn_j=(\ell'_j, \vn'_j)$ for $1\le j\le r$.

Now, let $D\in \mathfrak D^k_{r,u}$ be given. 
For each ${\tp{((\ell'_1, \vn'_1), \ldots, (\ell'_r, \vn'_r))}}\in \Phi^{(d)}(D,u)$, the rank of ${\tp{(\vn'_1, \ldots, \vn'_r)}}$ is either $r$ or $r-1$. 

Assume that $r\ge 2$. It is easy to verify that
\[\begin{split}
&\Phi^{(d)}(D,u)
=\left(\Lambda_D\times \Phi^{(d-1)}(D,u)\right)
\sqcup \\
&\hspace{0.1in}\bigsqcup_{w\in \NN} \bigsqcup_{C\in \mathfrak D^k_{r-1,w}(D)}
\hspace{-0.1in}(\Lambda_D-\Lambda_D(C))\times 
\left\{{\small\left(\begin{array}{c}
\vn'_1 \\
\vdots \\
\sum_{k=1}^{j_0-1} a_k\vn'_k \\
\vdots \\
\vn'_{r-1}\end{array}\right)_{r\times(d-1)}}: \left(\begin{array}{c}
\vn'_1 \\
\vdots \\
\vn'_{r-1}\end{array}\right)\in\Phi^{(d-1)}(C,w)\right\},
\end{split}\]
where $\mathfrak D^k_{r-1,w}(D)$ and $\Lambda_D(C)$ are as in Lemma~\ref{Lemma of MS Lemaa 7.7}.

Let us first compute the following integral
\begin{equation}\label{eq 1: new pf}\begin{split}
&\int_{\mathcal F'\times [0,1)^{d-1}} 
\sum_{\scriptsize \begin{array}{c}
{\tp{(\ell'_1, \ldots, \ell'_r)}}\\
\in \Lambda_D\end{array}}
\sum_{\scriptsize \begin{array}{c}
        {\tp{(\vn'_1, \ldots, \vn'_r)}}\\
\in \Phi^{(d-1)}(D,u)\end{array}}\\
&\hspace{0.4in} F\left(\begin{array}{c|c}
\left(\begin{array}{c}
z_1 \\ \vdots \\ z_1\end{array}\right)+ \dfrac D u \left(\begin{array}{c} \ell'_1 \\ \vdots \\ \ell'_r\end{array}\right)+ \dfrac D u \left(\begin{array}{c}
\vn'_1 \\ \vdots \\ \vn'_r\end{array}\right){\tp{\vv'}}
& \dfrac D u \left(\begin{array}{c}
\vn'_1 \\ \vdots \\ \vn'_r\end{array}\right)g'
\end{array}\right) \d\vv'\d\mu'(g').
\end{split}\end{equation}

Fix $g'\in \mathcal F'$ and $N:={\tp{(\vn'_1, \ldots, \vn'_r)}}\in \Phi^{(d-1)}(D,u)$. Set $J_N=\{1\le j_1<\ldots<j_r\le d-1\}$ such that 
$
N^{J_N}:=\left([N]^{j_1}, \ldots, [N]^{j_r}\right)
$
has a nonzero determinant. Denote by
                    \[
                        N\:{\tp{\vv'}}= N^{J_N}\:{\tp{\vv'_{J_N}}} + N^{J_N^c}\:{\tp{\vv'_{J_N^c}}},
                    \]
                    where $\vv'_{J_N}=(v_j)_{j\in J_N}\in \RR^r$ and $\vv'_{J_N^c}=(v_i)_{i\in J_N^c}\in \RR^{(d-1)-r}$.
Define
\[
G\left(\begin{array}{c}
x_1 \\ \vdots \\ x_r\end{array}\right)
:= \sum_{\scriptsize \begin{array}{c}
{\tp{(\ell'_1, \ldots, \ell'_r)}}\\
\in \Lambda_D\end{array}}
F\left(\begin{array}{c|c}
\left(\begin{array}{c} 
z_1 \\ \vdots \\ z_1\end{array}\right)
+\dfrac D u \left(\begin{array}{c} 
\ell'_1 \\ \vdots \\ \ell'_r\end{array}\right)
+\dfrac D u \left(\begin{array}{c} 
x_1 \\ \vdots \\ x_r\end{array}\right)
& \dfrac D u N g'
\end{array}\right).
\]
Obviously, $G$ is $\Lambda_D$-invariant so that it is $u\ZZ^r$-invariant, and $G(N\:\cdot)$ is $\ZZ^{d-1}$-invariant.  We want to compute the integral
\[
\int_{[0,1)^{d-1}} G\left(N\left(\begin{array}{c}
v_1 \\ \vdots \\ v_{d-1} \end{array}\right)\right) \d v_1 \cdots \d v_{d-1}
=\frac {1} {u^{d-1}} \int_{[0,u)^{d-1}}
G\left(N^{J_N} \vv^{J_N}+ N^{J_N^c} \vv^{J_N^c}\right) \d\vv^{J_N} \d\vv^{J_N^c}. 
\]
By the change of variables 
\[
\vv^{J_N} \mapsto N^{J_N} \vv^{J_N}+ N^{J_N^c} \vv^{J_N^c}={\tp{(x_1, \ldots, x_r)}},
\]
the above integral is
\[\begin{split}
&=\frac {1}{u^{d-1}} \int_{[0,u)^{d-1-r}}\int_{N^{J_N}[0,u)^r+N^{J_N^c}\vv^{J_N^c}}
G\left(\begin{array}{c}
x_1 \\ \vdots \\ x_r \end{array}\right) |\det N^{J_N}|^{-1} \d x_1 \cdots \d x_r \d\vv^{J_N^c}\\
&=\frac {1}{u^{d-1}} \int_{[0,u)^{d-1-r}}\int_{N^{J_N}[0,u)^r}
G\left(\begin{array}{c}
x_1 \\ \vdots \\ x_r \end{array}\right) |\det N^{J_N}|^{-1} \d x_1 \cdots \d x_r \d\vv^{J_N^c}\\
&=\frac {1} {u^r} \int_{[0,u)^r}
G\left(\begin{array}{c}
x_1 \\ \vdots \\ x_r \end{array}\right) \d x_1 \cdots \d x_r.
\end{split}\]

Let $\mathcal F_{\Lambda_D}$ be a fundamental domain for $\Lambda_D$ in $[0,u)^r$. Since $[0,u)^r$ is an $N(D,u)$-covering of $\mathcal F_{\Lambda_D}$, it follows that
\[\begin{split}
&\int_{[0,1)^{d-1}} G\left(N\left(\begin{array}{c}
v_1 \\ \vdots \\ v_{d-1} \end{array}\right)\right) \d v_1 \cdots \d v_{d-1}\\
&=\frac {N(D,u)} {u^r} 
\int_{\mathcal F_{\Lambda_D}} 
\sum_{\scriptsize \begin{array}{c}
{\tp{(\ell'_1, \ldots, \ell'_r)}}\\
\in \Lambda_D\end{array}}
F\left(\begin{array}{c|c}
\left(\begin{array}{c} 
z_1 \\ \vdots \\ z_1\end{array}\right)
+\dfrac D u \left(\begin{array}{c} 
x_1+\ell'_1 \\ \vdots \\ x_r+\ell'_r\end{array}\right)
& \dfrac D u N g'
\end{array}\right)\d x_1 \cdots \d x_r\\
&= \frac {N(D,u)} {u^r}
\int_{\RR^r} 
F\left(\begin{array}{c|c}
\left(\begin{array}{c} 
z_1 \\ \vdots \\ z_1\end{array}\right)
+\dfrac D u \left(\begin{array}{c} 
x_1\\ \vdots \\ x_r\end{array}\right)
& \dfrac D u N g'
\end{array}\right) \d x_1 \cdots \d x_r.
\end{split}\] 

Therefore, applying Proposition~\ref{Rogers component wise}, the integral \eqref{eq 1: new pf} is
\[\begin{split}
&\frac {N(D,u)} {u^r}
\int_{\SL_{d-1}(\ZZ)\backslash \SL_{d-1}(\RR)}
\sum_{\scriptsize \begin{array}{c}
        {\tp{(\vn'_1, \ldots, \vn'_r)}}\\
\in \Phi^{(d-1)}(D,u)\end{array}}\\
&\hspace{1in}\int_{\RR^r} F\left(\begin{array}{c|c}
\left(\begin{array}{c} 
z_1 \\ \vdots \\ z_1\end{array}\right)
+\dfrac D u \left(\begin{array}{c} 
x_1\\ \vdots \\ x_r\end{array}\right)
& \dfrac D u \left(\begin{array}{c}
\vn'_1 \\ \vdots \\ \vn'_r\end{array}\right)g'
\end{array}\right) \d x_1 \cdots \d x_r \d\mu'(g')\\
&=\frac {N(D,u)} {u^r} \cdot \frac {N(D,u)^{d-1}} {u^{(d-1)r}}\\
&\hspace{0.5in}\int_{(\RR^{d-1})^r}\int_{\RR^r}
F\left(\begin{array}{c|c}
\left(\begin{array}{c} 
z_1 \\ \vdots \\ z_1\end{array}\right)
+\dfrac D u \left(\begin{array}{c} 
x_1\\ \vdots \\ x_r\end{array}\right)
& \dfrac D u \left(\begin{array}{c}
\vx'_1 \\ \vdots \\ \vx'_r\end{array}\right)
\end{array}\right) \d x_1 \cdots \d x_r \d \vx'_1 \cdots \d \vx'_r\\
&= \frac {N(D,u)^d} {u^{dr}}
\int_{(\RR^d)^r} 
F\left(\sum_{i=1}^k z_1E_{i1} +\dfrac D u\left(\begin{array}{c}
\vx_1 \\ \vdots \\ \vx_r\end{array}\right)\right) \d \vx_1 \cdots \d \vx_r.
\end{split}\]       

\vspace{0.1in}
Now, let us fix $C\in \mathfrak D^k_{r-1,w}(D)$ and let $\Lambda_D(C)$ and $\mathcal R_D(C)$ be as in Lemma~\ref{Lemma of MS Lemaa 7.7}.
We want to compute 
\begin{equation}\label{eq 2: new pf}\begin{split}
&\int_{\mathcal F'\times [0,1)^{d-1}} 
\sum_{\scriptsize \begin{array}{c}
{\tp{(\ell'_1, \ldots, \ell'_r)}}\\
\in \Lambda_D-\Lambda_D(C)\end{array}}
\sum_{\scriptsize \begin{array}{c}
        {\tp{(\vn'_1, \ldots, \vn'_{r-1})}}\\
\in \Phi^{(d-1)}(C,w)\end{array}}\\
&\hspace{0.1in} F\left(\begin{array}{c|c}
\left(\begin{array}{c}
z_1 \\ \vdots \\ z_1\end{array}\right)+ \dfrac D u \left(\begin{array}{c} \ell'_1 \\ \vdots \\ \ell'_{r}\end{array}\right)+ \dfrac C w \left(\begin{array}{c}
\vn'_1 \\ \vdots \\ \vn'_{r-1}\end{array}\right){\tp{\vv'}}
& \dfrac C w \left(\begin{array}{c}
\vn'_1 \\ \vdots \\ \vn'_{r-1}\end{array}\right)g'
\end{array}\right) \d \vv'\d \mu'(g').\\
\end{split}\end{equation}

Since $\frac D u(\Lambda_D-\Lambda_D(C))= \frac D u (\mathcal R_D(C)+\Lambda_D(C))=\frac D u \mathcal R_D(C) + \frac C w \Lambda_C$ from Lemma~\ref{Lemma of MS Lemaa 7.7} (a), the integral \eqref{eq 2: new pf} is 
\begin{equation*}\begin{split}
&\sum_{\vl\in \mathcal R_D(C)}
\int_{\mathcal F'\times [0,1)^{d-1}} 
\sum_{\scriptsize \begin{array}{c}
{\tp{(\ell'_1, \ldots, \ell'_{r-1})}}\\
\in \Lambda_C\end{array}}
\sum_{\scriptsize \begin{array}{c}
        {\tp{(\vn'_1, \ldots, \vn'_{r-1})}}\\
\in \Phi^{(d-1)}(C,w)\end{array}}\\
&\hspace{0.3in} F\left(\begin{array}{c|c}
\left(\begin{array}{c}
z_1 \\ \vdots \\ z_1\end{array}\right)+ \dfrac D u \vl + \dfrac C w \left(\begin{array}{c} \ell'_1 \\ \vdots \\ \ell'_{r-1}\end{array}\right)+ \dfrac C w \left(\begin{array}{c}
\vn'_1 \\ \vdots \\ \vn'_{r-1}\end{array}\right){\tp{\vv'}}
& \dfrac C w \left(\begin{array}{c}
\vn'_1 \\ \vdots \\ \vn'_{r-1}\end{array}\right)g'
\end{array}\right) \d \vv'\d \mu'(g').
\end{split}\end{equation*}

Repeating the same argument with the above, where now we put $N={\tp{(\vn'_1, \ldots, \vn'_{r-1})}}$ with ${\tp{(\vn'_1, \ldots, \vn'_{r-1})}}\in \Phi^{(d-1)}(C,w)$ and 
\[
G\left(\begin{array}{c}
x_1 \\ \vdots \\ x_{r-1} \end{array}\right)
:=\hspace{-0.2in} \sum_{\scriptsize \begin{array}{c}
{\tp{(\ell'_1, \ldots, \ell'_{r-1})}}\\
\in \Lambda_C\end{array}}\hspace{-0.1in}
F\left(\hspace{-0.1in}\begin{array}{c|c}
\left(\begin{array}{c} 
z_1 \\ \vdots \\ z_1\end{array}\right)
+\dfrac D u \vl + \dfrac C w \left(\begin{array}{c} 
\ell'_1 \\ \vdots \\ \ell'_{r-1}\end{array}\right)
+\dfrac C w \left(\begin{array}{c} 
x_1 \\ \vdots \\ x_{r-1}\end{array}\right)
& \dfrac C w N g'
\end{array}\right),
\]                 
we have that the integral \eqref{eq 2: new pf} is
\[\begin{split}
\frac {N(C,w)^d} {w^{d(r-1)}}
\sum_{\scriptsize \begin{array}{c}
{\tp{(\ell_1, \ldots, \ell_k)}}\\
\in \frac D u \mathcal R_D(C)\end{array}}
\int_{(\RR^d)^{r-1}}
F\left(\sum_{i=1}^k (z_1+\ell_i)E_{i1} + \frac C w \left(\begin{array}{c}
\vx_1 \\ \vdots \\ \vx_{r-1}\end{array}\right)\right) \d \vx_1 \cdots \d \vx_{r-1}.
\end{split}\]

If $r=1$ and $\rk{\tp{(\vn'_1, \ldots, \vn'_r)}}=0$, that is, ${\tp{(\vn'_1, \ldots, \vn'_r)}}={\tp{(\origin, \ldots, \origin)}}$ and the integral is
\[
F\left(\sum_{i=1}^k (z_1 +\ell_i)E_{i1}\right),
\]
where ${\tp{(\ell_1, \ldots, \ell_k)}}\neq {\tp{(0, \ldots, 0)}}$.
Otherwise, they form $\Lambda_D\times \Phi^{(d-1)}(D,u)$ and one can proceed the same computation with the first case when $r\ge 2$.

Now the proposition follows from Lemma~\ref{Lemma of MS Lemaa 7.7} (c) after rearranging the summation with respect to $C\in \mathfrak D^{k}_{r-1, w}$ for $1\le r-1 \le k$ and $w\in \NN$.  
\end{proof}

For each $\vy\in \RR^d \smallsetminus \{\origin\}$, define
\[
X_q(\vy)=\left\{\Gamma(q) g\in X_q : \vy \in \left(\ZZ^d+ \frac {\vp} q \right)g\right\}.
\]
It is known that for each $t\in \NN$ with $\gcd(t,q)=1$, there is $\vk_t\in \ZZ^d+\vp/q$ with $\gcd(q\vk_t)=t$ so that we have the decomposition
\[
X_q(\vy)= \bigsqcup_{\scriptsize \begin{array}{c}
t\in \NN \\
(t,q)=1\end{array}}X_q(\vk_t, \vy),
\]
where $X_q(\vk_t, \vy) :=\{\Gamma(q)g\in X_q: \vk_t g=\vy\}$
(See \cite[Page 1993]{MS2010} for details).
Note that, the above decomposition holds for any such  choice of $\bm{k}_t$.
Moreover, if we put $g_t\in \SL_d(\RR)$ for each $t\in \NN$ with $\gcd(t,q)=1$ and $g_\vy\in \SL_d(\RR)$, respectively, such that $\ve_1 g_t=\vk_t$ and $\ve_1g_\vy=\vy$, it follows that 
\begin{equation}\label{def Xq(kt y)}
X_q(\vk_t, \vy)\simeq g_t^{-1}\left( (g_t\Gamma(q)g_t^{-1}\cap H)\backslash H\right) g_\vy\end{equation}
and one can define the probability measure $\nu_\vy$ on $X_q(\vy)$ for which $\nu_\vy|_{X_q(\vk_t, \vy)}$ is the pull-back measure of $\frac 1 {I_q\zeta(d)} \mu_H$, where $I_q:=[\SL_d(\ZZ):\Gamma(q)]$, with respect to the above identification (see \cite{MS2010}, especially (7.10)$\sim$(7.15) and Proposition 7.5).
                \begin{proposition}\label{MS Proposition 7.6}
                    Let $d\ge 3$ and $1\le k \le d-1$.
                    Suppose that ${\vp}\in \ZZ^d \smallsetminus \{\origin\}$ and $q\in \NN_{\ge 2}$ such that $\gcd(q, {\vp})=1$.
                    Let $P_t(\mathcal R(D))$ be as in Notation~\ref{notation:cong 1} after fixing $\mathcal R(D)$ for each $D\in \mathfrak D^k_{r,u}$.
                    
                    Let $F:(\RR^d)^k\rightarrow \RR_{\ge 0}$ be a bounded and compactly supported function.
                    For any $\vy\in \RR^d\smallsetminus \{\origin\}$, it follows that
                    \[\begin{split}
&\int_{X_q(\vy)}
\Siegel{k}{F} \left(\left(\ZZ^d+ \frac {\vp} q\right)g \right) \d \nu_{\vy}(g)\\
&=F\left({\tp{(\vy, \ldots, \vy)}}\right)+\sum_{\scriptsize \begin{array}{c}
        t\in \NN\\
        (t,q)=1\end{array}} \frac 1 {t^d} \sum_{\scriptsize \begin{array}{c}
        {\tp{(\ell_1, \ldots, \ell_k)}}\in \ZZ^k\\
        (\ell_1, \ldots, \ell_k, t)=1\end{array}}
F\left({\tp{\left(\frac {t+\ell_1 q} t \vy, \ldots, \frac{t+\ell_k} t \vy\right)}}\right)\\
&+\sum_{r=1}^{k-1}\sum_{u\in \NN} \sum_{D\in \mathfrak D^k_{r,u}}
\left[\frac {N(D,u)^d}{u^{dr}}
\int_{(\RR^d)^r} F\left(\left(\begin{array}{c} \vy \\ \vdots \\ \vy \end{array}\right) + \frac D u \left(\begin{array}{c} \vx_1 \\ \vdots \\ \vx_r\end{array}\right)\right) \d \vx_1 \cdots \d \vx_r\right.\\
&+\sum_{\scriptsize \begin{array}{c}
        t\in \NN\\
        (t,q)=1\end{array}}
        \sum_{\scriptsize 
        \begin{array}{c}
        \vl\in\\
        P_t(\mathcal R(D))\end{array}}
\frac {N(D,u)^{d}} {t^d\cdot u^{dr}}\times\\
&\hspace{0.8in}\left.\int_{(\RR^d)^r} F\left(
    \left(\begin{array}{c} \dfrac{t+\ell_1q}{t} \vy \\ \vdots \\ \dfrac{t+\ell_k q}{t} \vy\end{array}\right)
    + \frac D u \left(\begin{array}{c}
            \vx_1 \\
            \vdots\\
\vx_r\end{array}\right)\right) \d\vx_1 \cdots \d\vx_r\right]\\
&+\int_{(\RR^d)^k} F\left({\tp{(\vx_1, \ldots, \vx_k)}}\right)\d\vx_1 \cdots \d\vx_k.
                    \end{split}\]

                \end{proposition}
                \begin{proof}
                    Recall the definitions of $g_t$, $g_\vy$ as in \eqref{def Xq(kt y)}. If we let 
                    $a_{t/q}=\diag(t/q, q/t, 1, \ldots, 1),$
then one can further assume that $g_t=a_{t/q}\gamma_t$ for some $\gamma_t\in \SL_d(\ZZ)$ (\cite[Page 1993]{MS2010}).
                    By the definition of $\nu_{\vy}$ on $X_q(\vy)$,
                    \[\begin{split}
&\int_{X_q(\vy)}
\Siegel{k}{F} \left(\left(\ZZ^d+ \frac {\vp} q\right)g \right) \d \nu_{\vy}(g)\\
&=\frac 1 {I_q \zeta(d)} \sum_{\scriptsize \begin{array}{c}
        t\in \NN\\
(t,q)=1\end{array}}
\int_{(g_t \Gamma(q) g_t^{-1} \cap H)\backslash H}
\sum_{\scriptsize \begin{array}{c}
        \vm_i \in \ZZ^d \\
1\le i\le k\end{array}}
F\left(\left(\begin{array}{c}
            \vm_1+{\vp}/q \\
            \vdots \\
\vm_k+{\vp}/q \end{array}\right) g_t^{-1}h g_\vy\right) \d\mu_H (h)\\
&=\frac {q^d} {I_q \zeta(d)} \sum_{\scriptsize \begin{array}{c}
        t\in \NN\\
(t,q)=1\end{array}} \frac 1 {t^d}
\int_{(\Gamma(q)\cap H)\backslash H}
\sum_{\scriptsize \begin{array}{c}
        \vm_i \in \ZZ^d \\
1\le i \le k\end{array}}
F\left(\left(\begin{array}{c}
            \vm_1+{\vp}/q \\
            \vdots \\
\vm_k+{\vp}/q \end{array}\right) \gamma_t^{-1}h a_{t/q}^{-1}g_\vy\right) \d\mu_H (h).
                    \end{split}\]

Note that $(\ZZ^d+{\vp}/q)\gamma_t^{-1}= (\ZZ^d+\vk_t)\gamma_t^{-1}= \ZZ^d+ (t/ q) \ve_1$ and $(\Gamma(q)\cap H)\setminus H$ is a $(q^{d-1}I_q^{(d-1)})$-covering of $(\SL_d(\ZZ)\cap H)\setminus H$, where $I_q^{(d-1)}:=[\SL_{d-1}(\ZZ): \SL_{d-1}(\ZZ)\cap \Gamma(q)]$, one can apply Proposition~\ref{MS Lemma 7.7}. 
                    Since $E_{i1}a_{t/q}^{-1}=(q/t) E_{i1}$, the above expression equals
                    \[\begin{split}
 \frac {q^{2d-1}I_q^{(d-1)}} {I_q \zeta(d)}
\sum_{\scriptsize \begin{array}{c}
        t\in \NN\\
(t,q)=1\end{array}}
\frac 1 {t^d}&\left[\sum_{\ell_1, \ldots, \ell_k\in \ZZ}
F\left( {\tp{\left(\frac {t+\ell_1q}  t \vy, \ldots, \frac {t+\ell_k q} t \vy\right)}}\right)\right.\\
&+\sum_{r=1}^k \sum_{u\in \NN} \sum_{D\in \mathfrak D_{r,u}}\sum_{\scriptsize \begin{array}{c}
        \vl={\tp{(\ell_1,\ldots, \ell_k)}}\\
\in \mathcal R(D)\end{array}}
 \frac {N(D,u)^{d}} {u^{dr}}\times\\
&\hspace{-0.2in}\left.\int_{(\RR^d)^r} F\left(
         \left(\begin{array}{c} \dfrac{t+\ell_1q}{t}  \vy \\ \vdots \\ \dfrac{t+\ell_k q}{t} \vy\end{array}\right) +\frac D u \left(\begin{array}{c}
                \vx_1\\
                \vdots\\
\vx_r\end{array}\right)a_{t/q}^{-1} g_{\vy}\right)\d\vx_1\cdots \d\vx_r\right].
                    \end{split}\]

We will use the well-known fact that
\[
                        \frac  {I_q \zeta(d)}{q^{2d-1}I_q^{(d-1)}}
                        =\sum_{\scriptsize \begin{array}{c}
                                    t_1\in \NN\\
                        (t_1, q)=1\end{array}} \frac 1 {t_1^d}.
                    \]

For the first summation, which is the case when $r=0$, put $t=t_1\cdot t_2$, where $t_1=\gcd(\ell_1,\ldots, \ell_k,t)$. By renaming $(\ell_1/t_1, \ldots, \ell_k/t_1)$ by $(\ell_1, \ldots, \ell_k)$, it follows that
\[\begin{split}
&\frac {q^{2d-1}I_q^{(d-1)}} {I_q \zeta(d)}
\sum_{\scriptsize \begin{array}{c}
        t\in \NN\\
(t,q)=1\end{array}}
\frac 1 {t^d}\sum_{\ell_1, \ldots, \ell_k\in \ZZ}
F\left({\tp{\left(\frac {t+\ell_1q}  t \vy, \ldots, \frac {t+\ell_k q} t \vy\right)}}\right)\\
&=\frac {q^{2d-1}I_q^{(d-1)}} {I_q \zeta(d)}
\sum_{\scriptsize \begin{array}{c}
        t_1\in \NN\\
(t_1,q)=1\end{array}} \frac 1 {t_1^d}\left[F\left({\tp{\left(\vy, \ldots, \vy\right)}}\right)+\right.\\
&\hspace{1.4in}\left.
\sum_{\scriptsize \begin{array}{c}
        t_2\in \NN\\
(t_2,q)=1\end{array}} \frac 1 {t_2^d}\sum_{\scriptsize \begin{array}{c}
\ell_1, \ldots, \ell_k\in \ZZ\\
(\ell_1, \ldots, \ell_k,t_2)=1\end{array}}
F\left({\tp{\left(\frac {t_2+\ell_1q}  {t_2} \vy, \ldots, \frac {t_2+\ell_k q} {t_2} \vy\right)}}\right)\right]\\
&= F\left({\tp{\left(\vy, \ldots, \vy\right)}}\right)+\sum_{\scriptsize \begin{array}{c}
t\in \NN\\
(t,q)=1\end{array}} \frac 1 {t^d}
\sum_{\scriptsize\begin{array}{c}
\ell_1, \ldots, \ell_k\in \ZZ\\
(\ell_1, \ldots, \ell_k, t)=1\end{array}}
F\left({\tp{\left(\frac {t+\ell_1q}  t \vy, \ldots, \frac {t+\ell_k q} t \vy\right)}}\right).
\end{split}\]

For the case when $r=k$, we only have $u=1$, $D=\Id_k$ and $\mathcal R(D)=\{{\tp{(0,\ldots,0)}}\}$. After a change of variables, the integral in this case is
\[\begin{split}
&\frac {q^{2d-1} I_q^{(d-1)}} {I_q \zeta(d)}
\sum_{\scriptsize \begin{array}{c}
t\in \NN \\
(t,q)=1 \end{array}} \frac 1 {t^d} 
\int_{(\RR^d)^k} F\left( \left(\begin{array}{c} \vy \\ \vdots \\ \vy\end{array}\right) + \left(\begin{array}{c} \vx_1 \\ \vdots \\ \vx_k\end{array}\right) a^{-1}_{t/q} g_\vy \right) \d \vx_1 \cdots \d \vx_k\\
&=\frac {q^{2d-1} I_q^{(d-1)}} {I_q \zeta(d)}
\sum_{\scriptsize \begin{array}{c}
t\in \NN \\
(t,q)=1 \end{array}} \frac 1 {t^d} 
\int_{(\RR^d)^k} F \left(\begin{array}{c} \vx_1 \\ \vdots \\ \vx_k\end{array}\right)  \d \vx_1 \cdots \d \vx_k\\
&=\int_{(\RR^d)^k} F \left(\begin{array}{c} \vx_1 \\ \vdots \\ \vx_k\end{array}\right)  \d \vx_1 \cdots \d \vx_k.
\end{split}\]

Suppose that for $1\le r \le k-1$ and $u\in \NN$, $D\in \mathfrak D^k_{r,u}$ and $\mathcal R(D)$ are given. 
By rearranging the summation, it holds that
\[\begin{split}
&\frac {q^{2d-1} I_q^{(d-1)}} {I_q \zeta(d)}
 \sum_{\scriptsize \begin{array}{c}
t\in \NN \\
(t,q)=1 \end{array}} \frac 1 {t^d}
\sum_{\scriptsize \begin{array}{c}
        \vl={\tp{(\ell_1,\ldots, \ell_k)}}\\
\in \mathcal R(D)\end{array}}
 \frac {N(D,u)^{d}} {u^{dr}}\\
&\hspace{0.8in}\int_{(\RR^d)^r} F\left(
         \left(\begin{array}{c} \dfrac{t+\ell_1q}{t} \vy \\ \vdots \\ \dfrac{t+\ell_k q}{t} \vy\end{array}\right) +\frac D u \left(\begin{array}{c}
                \vx_1\\
                \vdots\\
\vx_r\end{array}\right)a_{t/q}^{-1} g_{\vy}\right)\d\vx_1\cdots \d\vx_r\\
\end{split}\]
\[\begin{split}
&=\frac {q^{2d-1} I_q^{(d-1)}} {I_q \zeta(d)}
 \sum_{\scriptsize \begin{array}{c}
t_1\in \NN \\
(t_1,q)=1 \end{array}} \frac 1 {t_1^d} \left[\frac {N(D,u)^d} {u^{dr}}
\int_{(\RR^d)^r} F\left(\left(\begin{array}{c} \vy \\ \vdots \\ \vy\end{array}\right) + \frac D u \left(\begin{array}{c} \vx_1 \\ \vdots \\ \vx_r\end{array}\right)\right) \d \vx_1 \cdots \d \vx_r\right.\\
&\left.+\hspace{-0.1in}
 \sum_{\scriptsize \begin{array}{c}
t\in \NN \\
(t,q)=1 \end{array}} \hspace{-0.1in}
\sum_{\scriptsize \begin{array}{c}
        \vl\in\\
P_t(\mathcal R(D))\end{array}}
\hspace{-0.15in}\frac 1 {t^d}
 \frac {N(D,u)^{d}} {u^{dr}}\int_{(\RR^d)^r} F\left(
        \left(\begin{array}{c} \dfrac{t+\ell_1q}{t} \vy \\ \vdots \\ \dfrac{t+\ell_k q}{t} \vy\end{array}\right) +\frac D u \left(\begin{array}{c}
                \vx_1\\
                \vdots\\
\vx_r\end{array}\right)\right)\d\vx_1\cdots \d\vx_r\right]\\
&=\frac {N(D,u)^d} {u^{dr}} \int_{(\RR^d)^r} F\left(\left(\begin{array}{c} \vy \\ \vdots \\ \vy\end{array}\right) + \frac D u \left(\begin{array}{c} \vx_1 \\ \vdots \\ \vx_r\end{array}\right)\right) \d \vx_1 \cdots \d \vx_r\\
&\hspace{0.1in}+\hspace{-0.1in}\sum_{\scriptsize \begin{array}{c}
t\in \NN \\
(t,q)=1 \end{array}}\hspace{-0.1in}
\sum_{\scriptsize \begin{array}{c}
        \vl\in\\
P_t(\mathcal R(D))\end{array}} \vspace{-0.15in}
\frac {N(D,u)^{d}} {t^d \cdot u^{dr}}\int_{(\RR^d)^r} F\left(
        \left(\begin{array}{c} \dfrac{t+\ell_1q}{t} \vy \\ \vdots \\ \dfrac{t+\ell_k q}{t} \vy\end{array}\right)+\frac D u \left(\begin{array}{c}
                \vx_1\\
                \vdots\\
\vx_r\end{array}\right)\right)\d\vx_1\cdots \d\vx_r,
\end{split}\]
 where for the first equality, as before, we put $t=t_1t_2$ with $t_1=\gcd(\vl,t)$ and rename $t_2$ and $\vl/t_1$ by $t$ and $\vl$, respectively.
This completes the proof of the proposition.
                \end{proof}

To prove Theorem~\ref{higher moment formula}, we need one more lemma which has appeared in \cite[(3.6)]{GKY2020} and also in \cite[(7.25)]{MS2010}.
                \begin{lemma}[{\cite[(3.6)]{GKY2020}}]\label{MS (7.25) GKY (3.6)}
                    For a Borel measurable function $\varphi: X_q\times \RR^d\rightarrow \RR_{\ge 0}$, we have
                    \[
                        \frac 1 {I_q} \int_{X_q} \sum_{\vm\in \ZZ^d}
                        \varphi\left(\Gamma(q)g, \left(\vm+\frac {\vp} q\right) g\right)\d\mu(g)
                        =\int_{\RR^d \smallsetminus \{\origin\}}\int_{X_q(\vy)} \varphi(\Gamma(q)g, \vy) \d \nu_{\vy}(g)\d\vy.
                    \]
                \end{lemma}

                \begin{proof}[Proof of Theorem~\ref{higher moment formula 1}]
                    Let $F : (\mathbb{R}^{d})^k \to \mathbb{R}_{\geq 0}$ be compactly supported bounded and
                    \[\begin{split}
                        \varphi(\Gamma(q)g, \vy)
&=\sum_{\scriptsize \begin{array}{c}
        \vm_i\in \ZZ^d\\
1\le i \le k-1\end{array}}
F\left(\vy, \left(\vm_1+\frac {\vp} q\right)g, \ldots, \left(\vm_{k-1}+\frac {\vp} q\right)g\right)\\
&=\Siegel{k-1}{F_\vy}\left(\left(\ZZ^d+\frac {\vp} q\right)g\right),
                    \end{split}\]
                    where $F_\vy:(\RR^d)^{k-1}\rightarrow \RR_{\ge0}$ is defined by
                    \[
                        F_\vy(\vy_2, \ldots, \vy_k)=F(\vy, \vy_2, \ldots, \vy_k).
                    \]
                    By Lemma~\ref{MS (7.25) GKY (3.6)},
                    we have that
                    \[\begin{split}
                        \frac 1 {J_q} \int_{Y_{{\vp}/q}} \Siegel{k}{F} (\Lambda) \d\mu(\Lambda)
&=\frac 1 {I_q} \int_{X_q} \sum_{\vm\in \ZZ^d}
\varphi\left(\Gamma(q)g,\left(\vm+\frac {\vp} q\right)g \right) \d\mu(g)\\
&=\int_{\RR^d \smallsetminus \{\origin\}}\int_{X_q(\vy)} \varphi(\Gamma(q)g, \vy) \d \nu_{\vy}(g)\d\vy\\
 &=\int_{\RR^d \smallsetminus \{\origin\}}\int_{X_q(\vy_1)}
 \Siegel{k-1}{F_{\vy_1}}(\Gamma(q)g) \d \nu_{\vy_1}(g) \d\vy_1.
\end{split}\]

For the first equality, let us recall that $X_q=\Gamma(q)\setminus \SL_d(\RR)$ is a $I_q/J_q$-covering of $Y_{\vp/q}$, where $I_q=[\SL_d(\ZZ):\Gamma(q)]$ and $J_q=[\SL_d(\ZZ):\Gamma_1(q)]$.

Thus for $F \geq 0$, the equations in Theorem \ref{higher moment formula 1} immediately follows from Proposition~\ref{MS Proposition 7.6}, where we replace $\vy_1$ by $\frac 1 t \vy_1$. 

Let us deal with the finiteness claim for $F \geq 0$: we will show that the LHS of \eqref{higher moment formula: eq (1)1} is finite.
             Define $\Phi : Y_{\bm{p}/q} \to X_d$ by $\Phi(\Lambda) := \Lambda - \Lambda$ for every $\Lambda \in Y_{\bm{p}/q}$.
             This map induces the measure $\Phi_*(\mu_q)$ on $X_d$, which is easily seen to be $\SL_d(\mathbb{R})$-invariant.
             Therefore $\Phi_*(\mu_q)$ equals to $\mu$ up to a positive constant.
             In fact $\Phi$ is the natural $J_q$-to-1 covering map from $Y_{\bm{p}/q}$ to $X_d$, thus $\Phi_*(\mu) = J_q\mu$.
             For $\Lambda \in Y_{\bm{p}/q}$ we have $\Lambda \subseteq q^{-1}\Phi(\Lambda)$.
             Therefore $\mathcal{S}_k(F)(\Lambda) \leq \mathcal{S}_k(F_q)(\Lambda)$, where $F_q : (\mathbb{R}^{d})^k \to \mathbb{R}_{\geq 0}$, $\bm{x} \mapsto F(q^{-1}\bm{x})$ is a compactly supported function.
             Hence
             \begin{align*}
                 \int_{Y_{\bm{p}/q}} \mathcal{S}_k(F)(\Lambda) \d{\mu_q(\Lambda)} &\leq \int_{Y_{\bm{p}/q}} \mathcal{S}_k(F_q)(\Phi(\Lambda)) \d{\mu_q(\Lambda)} \\
                 &= J_q\int_{X_d} \mathcal{S}_k(F_q)(\Lambda) \d{\mu(\Lambda)} < \infty,
             \end{align*}
             by Schmidt \cite{Sch58}.
             Thus in the present case, the sum in the RHS of \eqref{higher moment formula: eq (1)1} is convergent.

             We can now use classical techniques to prove Theorem \ref{higher moment formula 1} for a compactly supported bounded function $F : (\mathbb{R}^{d})^k \to \mathbb{R}$.
             We first note that Theorem \ref{higher moment formula 1} holds for $F_+ := \max(F, 0)$ and $F_- := \max(-F, 0)$.
             Finiteness for the function $|F|$ implies that the sum with $F$ in the RHS of \eqref{higher moment formula: eq (1)1} is absolutely convergent.
             Furthermore $\mathcal{S}_k(F)(\Lambda) = \mathcal{S}_k(F_+)(\Lambda) - \mathcal{S}_k(F_-)(\Lambda)$ (for a.e.\ $\Lambda$); we can integrate and rearrange to see that Theorem \ref{higher moment formula 1} holds.

\end{proof}

\subsection{Higher   moment formulae revisited}\label{Section 2.3}


For applications to Poisson distribution which are proved in the next section, we will need that the ``admissible matrices" appearing in the higher moment formula for $Y$ are contained in $\mathfrak D^{k'}_{r',u'}$ for some $k', r'$ and $u'$, which does not hold in Theorem \ref{higher moment formula: affine 1} and Theorem \ref{higher moment formula 1}. In the process of proving the needed variations of the higher moment formulae, we will define canonical sets of admissible matrices for each cases. In particular, we will see that the set of ``congruence-admissible matrices'' can be defined without using any choice of $\mathcal R(D)$ in Notation~\ref{notation:cong 1}.

Let us first refine the higher moment formula for the space $Y$ of affine lattices in $\RR^d$. 
\begin{theorem}\label{higher moment formula: affine}
Let $F:(\RR^d)^k\rightarrow \RR_{\ge 0}$ be bounded and compactly supported. 
For $d\ge 3$ and $3\le k\le d$,  
\begin{equation}\label{eq 1: higher moment formula: affine}
\int_Y \Siegel{k}{F}(\Lambda) \d\mu_Y(\Lambda)
=\sum_{m=1}^k \sum_{u\in \NN} \sum_{\widetilde{D}\in \mathfrak A^k_{m,u}}
\frac {N(\widetilde{D},u)^d} {u^{dm}}
\int_{(\RR^d)^m} F\left(\frac {\widetilde{D}} u \left(\begin{array}{c}
\vy_1 \\ \vdots \\ \vy_m \end{array}\right)\right) \d\vy_1 \cdots \d\vy_m,
\end{equation}
where $\mathfrak A^k_{m,u}$ is a subset of $\mathfrak D^k_{m,u}$ given by
\[
\mathfrak A^k_{m,u}=\left\{C\in \mathfrak D^k_{m,u} : \sum_{i=1}^{m} [C]^i={\tp{(u, \ldots, u)}} \right\}.
\]
\end{theorem}
Notice that when $m=1$ and $k$, the only possible $u$ is $u=1$ and 
\[
\mathfrak A^k_{1,1}=\left\{{\tp{(1, \ldots, 1)}}\right\}
\quad\text{and}\quad
\mathfrak A^k_{k,1}=\left\{\Id_k\right\},
\]
which corresponds to the first and second integrals of the RHS in \eqref{eq 1: higher moment formula: affine 1}, respectively.
\begin{proof}
Assume that $2\le m\le k-1$ so that $1\le r:=m-1\le k-2$.
Recall the $k\times m$ matrix $D'$ in Theorem~\ref{higher moment formula: affine 1} from $D\in \mathfrak D^{k-1}_{r,u}$. 

Take the map
\begin{equation}\label{transition: affine}
D\in \mathfrak D^{k-1}_{r,u}
\quad\mapsto\quad 
uD' 
\quad\mapsto\quad
\widetilde{D} \in \mathfrak D^k_{m,u},
\end{equation}
where we define $[\widetilde{D}]^1=[uD']^1 -\sum_{j=2}^m [uD']^j$ and $[\widetilde{D}]^j=[uD']^j$ for $2\le j\le m$. Clearly, the map is injective and $\widetilde D\in \mathfrak A^k_{m,u}$. 

Conversely, for any $C\in \mathfrak A^k_{m,u}$, denote by D the right-bottom minor of $C$ of the size $(k-1)\times r$. Then one can verify that $D\in \mathfrak D^{k-1}_{r,u}$ and $\widetilde D=C$.

Moreover, it is easy to show from their definitions that
\[
\frac {N(\widetilde{D},u)^d} {u^{d(r+1)}}
=\frac {N(D,u)^d} {u^{dr}},
\]
and the map $uD' \mapsto \widetilde D$ is the simple change of variables $\vy_{j}+\vy_1 \mapsto \vy'_{j}$ for $2\le j \le m$:
\[
\int_{(\mathbb{R}^{d})^{m}} F {\left({D'} \begin{pmatrix}
\bm{y}_1 \\ \vdots \\ \bm{y}_{m} \end{pmatrix}\right)} \d{\bm{y}_1} \cdots \d{\bm{y}_{m}}
=
\int_{(\mathbb{R}^{d})^{m}} F {\left(\frac{\widetilde D}{u}
\begin{pmatrix} \bm{y}_1 \\ \vdots \\ \bm{y}_{m} \end{pmatrix}\right)} \d{\bm{y}_1} \cdots \d{\bm{y}_{m}}.
\]
\end{proof}

In contrast to the affine case, in the congruence case it is difficult and complicated to describe the subset of matrices $\widetilde{D}\in \mathfrak D^k_{m,u_0}$, for given $1\le m \le k$ and $u_0\in \NN$, such that
\[
\widetilde{D}\RR^m= D'\RR^m \;\text{or}\; D'_{t,\vl}\RR^m
\]
for some $t\in \NN$ with $(t,q)=1$ and $\vl\in P_t(\mathcal R(D))$ appearing in Theorem~\ref{higher moment formula 1}.

For each $u\in \NN$ and $D\in \mathfrak D^{k-1}_{m-1,u}$, once we fix $\mathcal R(D)$ in Notation~\ref{notation:cong 1}, by the map
\begin{equation}\label{transition: cong}\begin{gathered}
D \mapsto uD'\mapsto \widetilde{D}\;\text{as in \eqref{transition: affine}}\\
(D, t, \vl) \mapsto u_0 D'_{t,\vl}=u_0\left(\begin{array}{c|c}
                        1 & 0 \, \cdots \, 0 \\
                        \hline
                        \begin{array}{c}
                    (t+\ell_1q)/t \\ \vdots \\ (t+\ell_{k-1}q)/t \end{array} & \dfrac 1 u D \end{array}\right) \mapsto \widetilde{D},
\end{gathered}\end{equation}
where $\widetilde D$ is defined by
\[
[\widetilde{D}]^1=[u_0D'_{t,\vl}]^1-\sum_{j=1}^{m-1} \frac {t+\ell_{i_j}q} t[u_0D'_{t,\vl}]^{j+1}
\quad\text{and}\quad
[\widetilde{D}]^j=[u_0D'_{t,\vl}]^j\;(j=2,\ldots, m),
\]
and $u_0\in \NN$ is taken such that $\widetilde {D}\in \Mat_{k,m}(\ZZ)$ with $\gcd \widetilde{D}=1$. Clearly, $\widetilde D\in \mathfrak D^k_{m, u_0}$.

Hence, one can attempt to define such a subset $\mathfrak C^k_{m,u_0}$ of $\mathfrak D^k_{m,u_0}$ by
\begin{equation}\label{def of mathfrak C 1}
\mathfrak C^k_{m,u_0}:=
\left\{C\in \mathfrak D^k_{m,u_0}: 
\begin{array}{c}
C=\widetilde{D}\;\text{for some}\;D\in \mathfrak D^{k-1}_{m-1,u}\;\text{or}\\
(D,t,\vl)\in \mathfrak D^{k-1}_{m-1,u}\times \NN\times \mathcal R(D)\;\text{in Notation~\ref{notation:cong 1}}\\
\text{defined as in \eqref{transition: cong}}
\end{array}
\right\}
\end{equation}
and reformulate the higher moment formula using these $\mathfrak C^k_{m,u_0}$.

As things stand,  $\mathfrak C^k_{m,u_0}$ seems to depend on an ad-hoc choice of a set of representatives $\mathcal{R}(D)$. However, the anonymous referee has kindly provided us with an argument using the Riesz representation theorem which shows that the set $\mathfrak C^k_{m,u_0}$ is independent to the choice of $\mathcal R(D)$ regardless of its role in the construction. With this as background, we now provide a cleaner definition of the set $\mathfrak C^k_{m,u_0}$, meaning that we don't need \emph{an ad-hoc choice} of $\mathcal R(D)$ for each $D\in \mathfrak D^{k-1}_{m-1,r}$. This definition was also suggested by the referee.


\begin{theorem}\label{higher moment formula}
    Let $d\ge 3$ and $1\le k \le d-1$.
    Let $F:(\RR^d)^k\rightarrow \RR_{\ge 0}$ be bounded and compactly supported. Then
    \begin{equation}\label{higher moment formula: eq (1)}\begin{split}
    &\int_{Y_{{\vp}/q}} \Siegel{k}{F} (\Lambda) \d\mu_q(\Lambda)
=\int_{(\RR^d)^k} F\left({\tp{(\vy_1, \ldots, \vy_k)}}\right) \d\vy_1 \cdots \d\vy_k\\
&\hspace{0.5in}+\sum_{m=1}^{k-1} \sum_{u\in \NN} \sum_{\widetilde {D}\in \mathfrak C^k_{m,u_0}}
\frac {N(\widetilde {D},u_0)^d} {u_0^{dm}} \int_{(\RR^d)^m} F\left(\frac {\widetilde D} {u_0} \left(\begin{array}{c}
\vy_1 \\ \vdots \\ \vy_m\end{array}\right)\right) \d\vy_1 \cdots \d\vy_m,
    \end{split}\end{equation}
where for $1\le m\le k-1$ and $u_0\in \NN$, 
\begin{equation}\label{def of mathfrak C 2}
\mathfrak C^k_{m,u_0}
=\left\{C\in \mathfrak D^k_{m,u_0}: \exists \vv=\left(\hspace{-0.05in}\begin{array}{c} v_1 \\ \vdots \\ v_k\end{array}\hspace{-0.05in}\right)\hspace{-0.05in} \in \frac C {u_0} \Lambda_C\;\;\text{s.t.}\begin{array}{c}
\gcd(v_1,q)=1,\\
v_1\equiv \cdots \equiv v_k\mod q,\;\text{and}\\
|v_1|=\min (\NN \cap \{\vv'\cdot \ve_1: \vv'\in \frac C u \Lambda_C \})\end{array} \right\}.
\end{equation}
Here, $\ve_1={\tp{(1,0,\ldots, 0)}}\in \RR^k$ and $\vv_1\cdot\vv_2={\tp{\vv_1}}\vv_2$ is the standard dot product of $\RR^k$.
\end{theorem}

\begin{proof}
We will consider the case when $m\ge 2$, then the case when $m=1$ would be easily seen. Let us first show that the sets defined as in \eqref{def of mathfrak C 1} and \eqref{def of mathfrak C 2} are identical.

Assume that $C=\widetilde{D}$ is an element of the set in \eqref{def of mathfrak C 1}. Note that
\[
\frac C {u_0} \Lambda_C = C\RR^m \cap \ZZ^k = D'\RR^m \cap \ZZ^k\;\text{or}\; D'_{t,\vl}\RR^m \cap \ZZ^k,
\]
where $D'$ or $D'_{t,\vl}$ is as in \eqref{transition: cong} for some $D\in \mathfrak D^{k-1}_{r,u}$ ($r:=m-1$), or $D\in \mathfrak D^{k-1}_{r,u}$, $t\in \NN$ with $(t,q)=1$, and $\vl={\tp{(\ell_1, \ldots, \ell_{k-1})}} \in \mathcal R(D)\subset \ZZ^{k-1}$ settled in Notation~\ref{notation:cong 1}, respectively. In particular,
\[
\vv:={\tp{(1, \ldots, 1)}}
\quad\text{or}\quad
{\tp{\left(t, t+\ell_1 q, \ldots, t+\ell_{k-1}q\right)}}
\in \frac C {u_0} \Lambda_C,\]
respectively. 

It suffices to show that
\begin{equation}\label{decom}
\frac C {u_0} \Lambda_C= \ZZ\vv \oplus {\tp{\left(0, \dfrac D u \Lambda_D \right)}},
\end{equation}
where ${\tp{\left(0, \frac D u \Lambda_D \right)}}$ is the embedded image of $\frac D u \Lambda_D\subseteq \RR^{k-1}$ into the last $(k-1)$ coordinates of $\RR^k$, since then it gives the fact that $v_1=\min (\NN \cap \{\vv'\cdot \ve_1: \vv'\in \frac C u \Lambda_C \})$.
The inclusion of the reverse direction is obvious. 

Pick an arbitrary $\vw\in \frac C {u_0} \Lambda_C$. Since $C\RR^m= \RR \vv \oplus {\tp{(0, D\RR^{r-1})}}$, one can take
\[\vw=\frac {c_1}{c_2}\vv+ {\tp{(0,\vv')}},\] 
where $c_1\in \ZZ$, $c_2\in \NN$ with $\gcd(c_1,c_2)=1$ and $\vv'\in D\RR^{r-1}\subseteq \RR^{k-1}$.
Since $\vw\in \ZZ^k$, it holds that
\begin{equation}\label{eq: cond 1}
\frac {c_1}{c_2}v_1 \in \ZZ
\;\Leftrightarrow\;
c_2|v_1
\quad\text{and}\quad 
\frac {c_1}{c_2}q\vl + \vv' \in \ZZ^{k-1}.
\end{equation}

If $v_1=1$, then automatically $c_2=1$ and $\vv'\in \ZZ^{k-1}\cap D\RR^r=\frac D u \Lambda_D$, which implies that $\vw\in \ZZ\vv \oplus {\tp{\left(0, \frac D u \Lambda_D\right)}}$.

Suppose that $v_1=t\ge 2$ so that $\vl\neq {\tp{(0,\ldots,0)}}$. Denote $\ell=\gcd(\vl)$ and $\widehat{\vl}=\frac 1 \ell \vl$, the primitive vector of the $\vl$-direction. 
Following Notation~\ref{notation:cong 1}, let $\vb_{k-m},\ldots, \vb_{k-1}$ be the basis of $\frac D u \Lambda_D$. Then it follows from the definition of $\mathcal R(D)$ that $\{\widehat{\vl}, \vb_{k-m}, \ldots, \vb_{k-1}\}$ is a primitive set, i.e.,
\[
\ZZ\widehat{\vl}\oplus \ZZ \vb_{k-m} \oplus \cdots \oplus \ZZ \vb_{k-1}
=\left(\RR\widehat{\vl}\oplus \RR \vb_{k-m} \oplus \cdots \oplus \RR \vb_{k-1}\right)
\cap \ZZ^{k-1}.
\]
Hence the second condition in \eqref{eq: cond 1} implies that
\[
\frac {c_1}{c_2} q \ell \in \ZZ
\;\Leftrightarrow\;
c_2 | q\ell
\quad\text{and}\quad
\vv'\in \frac D u \Lambda_D.
\]
Since $c_2|t$ from \eqref{eq: cond 1} as well and $\gcd(t,q\ell)=1$, we obtain the fact that $c_2=1$ and $\vw\in \RR\vv\oplus {\tp{\left(0, \frac D u \Lambda_D\right)}}$.
And this shows one inclusion.

\vspace{0.1in}
Conversely, let $C\in \mathfrak D^k_{m,u_0}$ be such that there is $\vv\in \frac C {u_0} \Lambda_C$ satisfying three conditions in \eqref{def of mathfrak C 2}. One can easily extract $D\in \mathfrak D^k_{r,u}$ from the right-bottom $(k-1)\times r$-minor of $C$ by making a primitive matrix which will be $D$, and $u$ is the unique nonzero entry of the first nonzero row of $D$. Fix any $\mathcal R(D)$.

Notice that the third condition is equivalent to saying that
\[
\frac C {u_0}\Lambda_C= \ZZ\vv \oplus \left(\frac C {u_0}\RR^m \cap {\tp{\left(0, \ZZ^{k-1}\right)}}\right)=\ZZ\vv{\tp{\left(0, \frac D u \Lambda_D\right)}}.
\]
Set $v_1=t$ (if $v_1<0$, replace $\vv$ by $-\vv$). From the first and second conditions, $\gcd(t,q)=1$ and $\vv={\tp{(t, t+\ell'_1q, \ldots, t+\ell'_{k-1}q)}}$ for some $\vl'={\tp{(\ell'_1, \ldots, \ell'_{k-1})}}\in \ZZ^{k-1}$.
Since $\mathcal R(D)\simeq \ZZ^{k-1}/\frac D u \Lambda_D$, there is the unique $\vl={\tp{(\ell_1, \ldots, \ell_{k-1})}}\in \mathcal R(D)$ for which $\vl+\frac D u \Lambda_D=\vl'+\frac D u \Lambda_D$ and
\[
\frac C {u_0} \Lambda_C=\ZZ \left(\begin{array}{c} t \\ t+\ell_1q \\ \vdots \\ t+\ell_{k-1}q \end{array}\right) \oplus \left(\begin{array}{c} 0 \\ \\ \dfrac D u \Lambda_D \\[0.1in] \end{array}\right).
\]
This shows that ${\tp{(t, t+\ell_1q, \ldots, t+\ell_{k-1}q)}}$ is primitive. If $\vl={\tp{(0,\ldots, 0)}}$, then it holds that $C=\widetilde{D}$ of the first type described in \eqref{def of mathfrak C 1}. If $\vl\neq{\tp{(0,\ldots, 0)}}$, then $\gcd(\vl, t)=1$ so that $\vl\in P_t(\mathcal R(D))$ and $C=\widetilde{D}$ defined from $(D,t,\vl)$.

\vspace{0.1in}
Now, to establish the theorem, considering the change of variables on $\vy_1$ in Theorem~\ref{higher moment formula 1}, it is left to show that 
\[
\frac {N(\widetilde D, u_0)^d}{u_0^{dm}}
=\frac {N(D, u)^d}{t^d \cdot u^{dr}},
\]
where we put $t=1$ when $\widetilde{D}\in \mathfrak C^k_{m,u_0}$ is of the first type in \eqref{def of mathfrak C 1}. Recall that $N(\widetilde D, u_0)$ is the number of integral solutions $\vz={\tp{(z_1, \ldots, z_m)}}\in \ZZ^m$ modulo $u_0$ for which $\frac {\widetilde{D}} {u_0} \vz \in \ZZ^k$. Equivalently, $N(\widetilde D, u_0)$ is the number of integral solutions $\vz\in \ZZ^m$ modulo $u_0$ for which $D'\vz\in \ZZ^k$ or $D'_{t,\vl}\vz\in \ZZ^k$, respectively.

Based on \eqref{decom}, it follows that $z_1\in t\ZZ$ and there are $(u_0/t)$-number of such $z_1\in \ZZ$ modulo $u_0$.
Moreover, as long as $z_1\in t\ZZ$, $t[D']^1\in \ZZ^{k}$ and we reduce that 
\[\frac D u {\tp{(z_2, \ldots, z_k)}}\in \ZZ^{k-1}
\]
modulo $u_0$, and the number of such ${\tp{(z_2, \ldots, z_k)}}$ is $(u_0/u)^r(N(D,u))$. Therefore
\[
\frac {N(\widetilde D, u_0)^d}{u_0^{dm}}
=\frac {(u_0/t)^d \cdot (u_0/u)^{dr} N(D,u)^d} {u_0^{dm}}
=\frac {N(D,u)^d}{t^d \cdot u^{dr}}.
\]
\end{proof}

\section{Poissonian Behaviour}\label{Poisson}
\subsection{Affine Case}
In this section, we prove Theorem \ref{thmppmainaffine}. Recall that for each $d \geq 2$, we set $\mathcal{S} = \{S_t: t \geq 0\}$ be an increasing family of subsets of $\mathbb{R}^{d}$ with $\vol(S_t) = t$, and for $\Lambda \in Y$ set
\[
    N_t(\Lambda) := \#{\left(S_t \cap \Lambda\right)}.
\]
Denote by $\{N^\lambda(t) : t \geq 0\}$ a Poisson process on the non-negative real line with intensity $\lambda$.


For $\Lambda \in Y$ we order the lengths of vectors in $\Lambda$ as $0\leq \ell_1 \leq \ell_2 \leq \ell_3 \leq \cdots $, and let $\mathscr{V}_i$ denote the volume of the closed ball of radius $\ell_i$ centered at origin.
If we take $\mathcal{S} = \{B_t: t \geq 0\}$ to be the family of closed balls with $\vol(B_t) = t$ around origin, then
\[
    N_t(\Lambda) = \#\{i : \mathscr{V}_i \leq t\}.
\]
In this specific case Theorem \ref{thmppmainaffine} is equivalent to:
\begin{theorem}[]\label{thmppequi}
    For any fixed $n$, the $n$-dimensional random variable $(\mathscr{V}_1, \ldots, \mathscr{V}_n)$ converges in distribution to the distribution of the first $n$ points of a Poisson process on the non-negative real line with intensity $1$ as $d \to \infty$.
\end{theorem}

In this form the theorem determines the limit distribution of lengths of vectors in a random lattice as $d \to \infty$.

We will now prove the above general Theorem \ref{thmppmainaffine} by proving a joint moment formula for $N_t(\cdot)$.
Let $k \geq 1$ and $0 \leq V_1 \leq \cdots \leq V_k$.
We use, by abuse of notation, $N_i( \cdot )$ to denote $N_{V_i}( \cdot )$.
Note that $N_i( \cdot ) = \widehat{\rho_i}( \cdot )$, where $\rho_i$ is the characteristic function of $S_{V_i}$.
We calculate, following S\"{o}dergren \cite{Sod}, the `main term' of the joint moment of $N_i$'s.
In this regard we apply Theorem \ref{higher moment formula: affine} with $F = \prod_{i = 1}^{k}\rho_i$ defined as
\[
    F \begin{pmatrix}
        \bm{y}_1 \\
        \vdots \\
        \bm{y}_k
    \end{pmatrix} = \prod_{i = 1}^{k} \rho_i(\bm{y}_i).
\]
We consider the sub-collection of the RHS of \eqref{eq 1: higher moment formula: affine} consisting of terms corresponding to $m = 1$ and $k$, and terms from the sum corresponding to $u = 1$ and $\widetilde{D} \in \mathfrak{A}_{m, 1}^{k}$ satisfying that $\widetilde{D}$ has exactly one non-zero entry in each row, with all non-zero entries of $\widetilde{D}$ being of modulus 1.
The set of such matrices $\widetilde{D} \in \mathfrak{A}_{m, 1}^{k}$ is $\mathfrak{M}^k$, where
\begin{align*}
    \mathfrak{R}_1^k &= \bigcup_{2 \leq m \leq k-1} {\left({\left(\bigcup_{u \geq 2} \mathfrak{A}_{m, u}^k\right)} \cup {\left\{\widetilde{D} = {\left(\widetilde{D}_{ij}\right)} \in \mathfrak{A}_{m, 1}^k : \exists |\widetilde{D}_{ij}| \geq 2\right\}}\right)}, \\
    \mathfrak{R}_{2}^k &= {\left\{\widetilde{D} \in {\left(\bigcup_{2 \leq m \leq k-1} \mathfrak{A}_{m, 1}^k\right)} \smallsetminus \mathfrak{R}_1^k: \substack{\displaystyle\exists~\text{row such that at least} \\ \\ \displaystyle\text{two entries are non-zero}}\right\}}, \\
    \mathfrak{M}^k &= {\left(\bigcup_{2 \leq m \leq k-1} \mathfrak{A}_{m, 1}^k \smallsetminus {\left(\mathfrak{R}_1^k \cup \mathfrak{R}_{2}^k\right)}\right)} \cup \left\{\mathrm{Id}_k, \begin{pmatrix}
        1 \\
        \vdots \\
        1
    \end{pmatrix}\right\}.
\end{align*}
Here, we want to mention that we will use the same notations $\mathfrak R^k_1$, $\mathfrak R^k_2$ and $\mathfrak M^k$ for the analogous (but different) sets in each subsection (see Subsection \ref{subsection Congruence case} and Section \ref{section Limit in distribution}). This will hopefully cause no confusion.

We denote this sub-collection of the RHS of \eqref{eq 1: higher moment formula: affine} as $M_{d, k}^{\text{affine}}$ and the rest of the terms as $R_{d, k}^{\text{affine}}$.
That is,
\[
    \mathbb{E}{\left(\prod_{i = 1}^{k} N_i\right)} = M_{d, k}^{\text{affine}} + R_{d, k}^{\text{affine}}
\]
where
\begin{equation}\label{eqmdkaffine}
        M_{d, k}^{\text{affine}} = \sum_{\widetilde{D} \in \mathfrak{M}^{k}} \int_{(\mathbb{R}^{d})^{m}} \prod_{i= 1}^{k}\rho_i{\left(\widetilde{D} \begin{pmatrix}
                        \vy_1 \\
                        \vy_2 \\
                        \vdots \\
                        \vy_{m}
            \end{pmatrix}\right)} \d{\vy_1}\d{\vy_2} \cdots \d{\vy_{m}}
\end{equation}
and
\begin{equation}\label{eqrdkaffine}
    R_{d, k}^\text{{affine}} = \sum_{\widetilde{D} \in \mathfrak{R}_1^{k} \cup \mathfrak{R}_2^{k}} \frac{N(\widetilde{D}, u)^d}{u^{dm}} \int_{(\mathbb{R}^{d})^{m}} \prod_{i= 1}^{k}\rho_i{\left(\frac{1}{u}\widetilde{D} \begin{pmatrix}
                        \vy_1 \\
                        \vy_2 \\
                        \vdots \\
                        \vy_{m}
            \end{pmatrix}\right)} \d{\vy_1}\d{\vy_2} \cdots \d{\vy_{m}}.
\end{equation}

Let $(\alpha, \beta)$ be a division of $\{1, \ldots, k\}$, i.e., $\alpha = \{\alpha_1 < \cdots < \alpha_m\}$ and $\beta = \{\beta_1 < \cdots < \beta_{k - r}\}$ are complementary subsets of $\{1, \ldots, k\}$ with $\alpha \neq \varnothing$.
Define
\[
    \mathfrak{M}_{\alpha, \beta}^{\text{affine}} := \left\{\widetilde{D} \in \mathfrak{M}^k: I_{\widetilde{D}} = \alpha\right\}
\]
and let $M_{\alpha, \beta}^{\text{affine}}$ denote the cardinality of $\mathfrak{M}_{\alpha, \beta}^{\text{affine}}$.
We allow for the case $(\alpha, \beta) = {\left(\{1, \ldots, k\}, \varnothing\right)}$, in which case $\mathfrak{M}_{\alpha, \beta}^k = \{\mathrm{Id}_k\}$.
Thus we can rewrite \eqref{eqmdkaffine} as:
\begin{equation}\label{eqmdkaffinenew}
    M_{d, k}^{\text{{affine}}} = \sum_{(\alpha, \beta)} \sum_{\widetilde{D} \in \mathfrak{M}_{\alpha, \beta}^{\text{affine}}} \int_{(\mathbb{R}^{d})^{m}} \prod_{i= 1}^{k}\rho_i{\left(\widetilde{D} \begin{pmatrix}
                        \vy_1 \\
                        \vy_2 \\
                        \vdots \\
                        \vy_{m}
            \end{pmatrix}\right)} \d{\vy_1}\d{\vy_2} \cdots \d{\vy_{m}}.
\end{equation}
where the outer sum is over all possible divisions $(\alpha, \beta)$ of $\{1, \ldots, k\}$.

\begin{remark}\label{rementryonly1}
    It follows from the definition of $\widetilde{D}$ that for $\widetilde{D} \in \mathfrak{M}^k$, the non-zero entries of the matrix $\widetilde{D}$ can only be 1.
Since $\widetilde{D} \notin \mathfrak R^k_1$, we already know that entries of $\widetilde{D}\in \{0, \pm 1\}$. 
The fact that $-1$ is not possible for entries of $\widetilde{D}$ comes from notations in Theorem \ref{higher moment formula: affine}.
Suppose that there is a row having $-1$ in its entries. Let $(x_1, x_2, \ldots, x_{m})$ be such a row.
If $x_1=-1$, since $x_1=1 - \sum_{\ell=2}^{m} x_\ell$, there should be at least one nonzero element in $(x_2, \ldots, x_{m})$, which contradicts to the fact that each row, there is only one nonzero entry. One can also obtain a contradiction when one assumes that there is some $2\le i_0\le m$ for which $x_{i_0}=-1$.

\end{remark}

\begin{lemma}[]\label{thmsodmainterm}
    With notations as above
    \begin{equation}\label{eqsodmainterm}
        M_{d, k}^{\text{\emph{affine}}} = \sum_{(\alpha, \beta)} M_{\alpha, \beta}^{\text{\emph{affine}}} \prod_{i = 1}^{m} V_{\alpha_i}.
    \end{equation}
\end{lemma}
\begin{proof}
    Consider any matrix $\widetilde{D} = \left(\widetilde{D}_{ij}\right) \in \mathfrak{M}_{\alpha, \beta}^\text{{affine}}$ and let $\lambda_\ell$ be such that $\widetilde{D}_{\beta_\ell, \lambda_\ell} = 1$, $1 \leq \ell \leq k - m$.
    Then, as $V_i$'s are increasing, the following calculation finishes the proof
    \begin{align*}
        &\int_{(\mathbb{R}^{d})^{m}} \prod_{i = 1}^{k} \rho_i\left(\widetilde{D}\begin{pmatrix}
                \vy_1 \\
                \vy_2 \\
                \vdots \\
                \vy_{m}
        \end{pmatrix}\right) \d{\vy_1} \d{\vy_2} \cdots \d{\vy_{m}} 
        = \int_{(\mathbb{R}^{d})^{m}} \prod_{i = 1}^k \rho_i{\left(\sum_{j = 1}^{m} \widetilde{D_{ij}} \vy_j\right)} \d{\vy_1} \cdots \d{\vy_{m}} \\
=& \int_{(\mathbb{R}^{d})^{m}} \prod_{i = 1}^{m} \rho_{\alpha_i}(\vy_i) \prod_{\ell = 1}^{k - m} \rho_{\beta_\ell}(\bm{y}_{\lambda_\ell}) \d{\vy_1} \cdots \d{\vy_{m}} 
= \int_{(\mathbb{R}^{d})^m} \prod_{i = 1}^{m} \rho_{\alpha_i}(\bm{y}_i) = \prod_{i = 1}^{m} V_{\alpha_i}.
    \end{align*}
\end{proof}

We shall now mention some estimates regarding $R_{d, k}^{\text{affine}}$.
These estimates are originally due to Rogers \cite{Rogers55-2, Rogers56-2} and they were generalized to Lemma \ref{lemestimatesfromrogers} (below) by S\"{o}dergren \cite{Sod}.
For $D \in \mathfrak{D}_{r, u}^k$ set
\[
    I(D, u) := \int_{(\mathbb{R}^{d})^r} \prod_{i = 1}^{k} \rho_i{\left(\frac{1}{u}D \begin{pmatrix}
                \vy_{1} \\
                \vdots \\
                \vy_{r}
    \end{pmatrix}\right)} \d{\vy_1} \cdots \d{\vy_{r}}.
\]
\begin{lemma}[Estimates from \cite{Rogers55-2}, \cite{Rogers56-2} and \cite{Sod}]\label{lemestimatesfromrogers}
    For $d > [k^2/4] + 2$
    \begin{enumerate}[(i)]
        \item $\displaystyle\sum_{r = 1}^{k} \sum_{u = 2}^{\infty} \sum_{D \in \mathfrak{D}_{r, u}^k} \frac{N(D, u)^d}{u^{dr}} \cdot I(D, u) = O\left(2^{-d}\right)$, \\
        \item $\displaystyle\sum_{r = 1}^{k} \sum_{D \in \mathfrak{D}_{r, 1}^{k, 1}} I(D, u) = O\left(2^{-d}\right)$, where $\mathfrak{D}_{r, 1}^{k, 1} \subseteq \mathfrak{D}_{r, 1}^k$ contains matrices $D$ such that $\max |d_{ij}| \geq 2$, \\
        \item $\displaystyle\sum_{r = 1}^{k} \sum_{D \in \mathfrak{D}_{r, 1}^{k, 2}} I(D, u) = O\left({\left(3/4\right)}^{d/2}\right)$, where $\mathfrak{D}_{r, 1}^{k, 2} \subseteq \mathfrak{D}_{r, 1}^k$ contains matrices $D$ such that $\max |d_{ij}| = 1$ and at least one row of $D$ has at least two non-zero entries.
    \end{enumerate}
\end{lemma}

\begin{proof}
    The proof of this lemma can be found in \cite[Proposition 2, Lemma 1 and Lemma 2]{Sod}.
    The main ingredients in S\"odergren's proof are \cite[\S9]{Rogers55-2} and the contents of \cite[\S4]{Rogers56-2}.

    We remark that we only need the fact that $N(D,u)^d/u^{dr}\le 1/u^d$ to prove Property (i). Hence, one can use Lemma~\ref{lemestimatesfromrogers} for applications with the space $Y_{\vp/q}$ as well as the space $Y$ in Section~\ref{Poisson} and Section~\ref{section Limit in distribution}.
\end{proof}

Rogers' estimate shows that:
\begin{lemma}[]\label{thmremainderofaffine}
    \[
        R_{d, k}^{\text{\emph{affine}}} = O{\left({\left(3/4\right)}^{d/2}\right)}.
    \]
\end{lemma}
\begin{proof}
    It follows from \eqref{eqrdkaffine} that $R_{d, k}^{\text{affine}}$ is less than the sum of LHS's in Lemma \ref{lemestimatesfromrogers}.
\end{proof}

The lemmas above combine to give us the following theorem:
\begin{theorem}[]\label{thmexpniconv}
    \begin{equation}\label{eqexpniconv}
        \mathbb{E}{\left(\prod_{i = 1}^{k} N_i\right)} \to \sum_{(\alpha, \beta)} M_{\alpha, \beta}^{\text{\emph{affine}}} \prod_{i = 1}^{r} V_{\alpha_i}
    \end{equation}
    as $d \to \infty$.
\end{theorem}

\subsubsection{Proof of Theorem \ref{thmppmainaffine}}
This proof closely follows the proof of Theorem 1 in \cite[\S4]{Sod}.
Let us discuss the Poisson process $\{N^\lambda(t): t \geq 0\}$.
By definition $N^\lambda(t)$ denotes the number of points falling in the interval $[0, t]$ and $N^\lambda(t)$ is Poisson distributed with expectation $\lambda t$.
By $0 \leq T_1 \leq  T_2 \leq T_3 \leq \cdots$ let us denote the points of the Poisson process.

\begin{lemma}[]\label{lemexppp}
    Let $k \geq 1$ and let $\mathscr{P}(k)$ denote the set of partitions of $\{1, \ldots, k\}$.
    For $1 \leq i \leq k$ let $f_i : \mathbb{R}_{\geq 0} \to \mathbb{R}$ be functions satisfying $\prod_{i \in B} f_i \in L^1(\mathbb{R}_{\geq 0})$ for every nonempty subset $B \subseteq \{1, \ldots, k\}$.
    Then
    \begin{equation}\label{eqexppp}
        \mathbb{E}{\left(\prod_{i = 1}^{k} {\left(\sum_{\ell = 1}^{\infty} f_i(T_\ell)\right)}\right)} = \sum_{P \in \mathscr{P}(k)} \lambda^{\#P} {\left(\int_{0}^{\infty} \prod_{i \in B} f_i(x) \d{x}\right)}.
    \end{equation}
\end{lemma}
\begin{proof}
    The proof of this lemma is similar to \cite[Proposition 3]{Sod}.
\end{proof}

We apply Lemma \ref{lemexppp} with functions $f_i = \mychi_i, 1 \leq i \leq k$, where $\mychi_i$ is the characteristic function of the interval $[0, V_i]$.
Thus we get
\begin{equation}\label{eqexpofpp}
    \begin{split}
        \mathbb{E}{\left(\prod_{i = 1}^{k} N^\lambda(V_i)\right)} &= \mathbb{E}{\left(\prod_{i = 1}^{k} {\left(\sum_{\ell = 1}^{\infty} \mychi_i(T_\ell)\right)}\right)} \\
                                                                   &= \sum_{P \in \mathscr{P}(k)} \lambda^{\#P} \prod_{B \in P} {\left(\int_{0}^{\infty} \prod_{i \in B} \mychi_i(x) \d{x}\right)} \\
                                                                   &= \sum_{P \in \mathscr{P}(k)} \lambda^{\#P} \prod_{B \in P} V_{i_B},
    \end{split}
\end{equation}
where $i_B = \min_{i \in B} i$.

The following lemma helps us compare the RHS's of \eqref{eqexpniconv} and \eqref{eqexpofpp} for $\lambda = 1$.
\begin{lemma}[\cite{Sod}, Lemma 3]\label{lemsodlem3}
    There is bijection $g : \mathfrak{M}^k \to \mathscr{P}(k)$ with the property that if $\widetilde{D} \in \mathfrak{M}^k$ is an $k \times m$ matrix and $g(\widetilde{D}) = P = \{B_1, \ldots, B_{\#P}\}$ then $\#P = m$ and $\{\alpha_1 < \cdots < \alpha_m\} = \{i_{B_1} < \cdots < i_{B_m}\}$.
\end{lemma}
\begin{proof}
    Other than switching the rows and columns of the matrices $D$, the proof of this lemma is same as \cite[Lemma 3]{Sod}.
\end{proof}

Theorem \ref{thmexpniconv}, \eqref{eqexpofpp} and Lemma \ref{lemsodlem3} imply the following result:
\begin{theorem}[]\label{thmexpniconvi}
    \[
        \mathbb{E}{\left(\prod_{i = 1}^{k} N_i\right)} \to \mathbb{E}{\left(\prod_{i = 1}^{k} N^1(V_i)\right)}
    \]
    as $d \to \infty$.
\end{theorem}
\begin{corollary}[]\label{corconvindist}
    Let $\bm{V} = (V_1, \ldots, V_k)$ and consider the random vectors
    \[
        \bm{N}(\Lambda, \bm{V}) = {\left(N_1(\Lambda), \ldots, N_k(\Lambda)\right)}
    \]
    and
    \[
        \bm{N}(\bm{V}) = {\left(N^1(V_1), \ldots, N^1(V_k)\right)}.
    \]
    Then $\bm{N}(\Lambda, \bm{V})$ converges in distribution to $\bm{N}(\bm{V})$ as $d \to \infty$.
\end{corollary}
\begin{proof}
    This proof follows a similar line of argument as \cite[Corollary 1]{Sod}.
    We omit it for the sake of brevity.
\end{proof}

Corollary \ref{corconvindist} implies that all finite dimensional distributions coming from the process $\{N_t(\Lambda) : t \geq 0\}$ converge to the corresponding finite dimensional distributions of the Poisson process $\{N^1(t) : t \geq 0\}$ as $d \to \infty$.
By \cite[Theorem 12.6 and Theorem 16.7]{Bil}, we see that the process $\{N_t(\Lambda) : t \geq 0\}$ converges weakly to the process $\{N^1(t) : t \geq 0\}$ as $d \to \infty$. \\


Corollary \ref{corconvindist}, with $k = 1$, is a generalization of \cite[Theorem 3]{Rogers56-2} to the affine case.

\subsection{Congruence Case}\label{subsection Congruence case}
In this section, we prove Theorem \ref{thmppmaincong}. We recall the notation. For $d \geq 2$ let $\mathcal{S} = \{S_t: t > 0\}$, an increasing family of subsets of $\mathbb{R}^{d}$ and $\bm{p}/q \in \mathbb{Q}^{d}$.
For $\Lambda \in Y_{\bm{p}/q}$ set
\[
    N_t(\Lambda) = \#(S_t \cap \Lambda).
\]

    

For $\Lambda \in Y_{\bm{p}/q}$ let us order the lengths of non-zero vectors in $\Lambda$ as $0 < \ell_1 \leq \ell_2 \leq \ell_3 \leq \cdots $, and let $\mathscr{V}_i$ denote the volume of the closed ball of radius $\ell_i$ centered at origin.
Taking $\mathcal{S} = \{B_t: t > 0\}$ to be the increasing family of closed balls with $\vol(B_t) = t$ around origin we see that
\[
    N_t(\Lambda) = \#\{i : \mathscr{V}_i \leq t\}.
\]
Thus Theorem \ref{thmppmaincong} is equivalent to:
\begin{theorem}[]\label{thmppmaincong2}
    For and any fixed $n$, the $n$-dimensional random variable $(\mathscr{V}_1, \ldots, \mathscr{V}_n)$ converges in distribution to the distribution of the first $n$ points of a Poisson process on the non-negative real line with intensity
    \[
        \begin{dcases}
            1 & \text{if} \ q \geq 3, \\
            \tfrac{1}{2} & \text{if} \ q = 2. \\
        \end{dcases}
    \]
\end{theorem}

As in the affine case we approach Theorem \ref{thmppmaincong} via a joint moment formula for $N_t( \cdot )$.
Let $k \geq 1$ and $0 < V_1 \leq V_2 \leq \cdots \leq V_k$.
Define $N_i$'s, $\rho_i$'s and $F$ similar to the affine case.
We apply Theorem \ref{higher moment formula} to the function $F$.
We first consider the sub-collection of the RHS of \eqref{higher moment formula: eq (1)} denoted by $M_{d, k}^{\text{cong}}$, defined as
\begin{equation}\label{eqmdkcong}
    M_{d, k}^{\text{cong}} := \sum_{\widetilde{D} \in \mathfrak{M}^k} \int_{(\mathbb{R}^{d})^m} \prod_{i = 1}^{k} \rho_i {\left(\widetilde{D} \begin{pmatrix}
                \bm{y}_1 \\
                \vdots \\
                \bm{y}_m
    \end{pmatrix}\right)} \d{\bm{y}_1} \cdots \d{\bm{y}_m}
\end{equation}
where
\begin{align*}
    \mathfrak{R}_1^k &= \bigcup_{1 \leq m \leq k-1} {\left({\left(\bigcup_{u \geq 2} \mathfrak{C}_{m, u}^k\right)} \cup {\left\{\widetilde{D} = {\left(\widetilde{D}_{ij}\right)} \in \mathfrak{C}_{m, 1}^k : \exists |\widetilde{D}_{ij}| \geq 2\right\}}\right)}, \\
    \mathfrak{R}_{2}^k &= {\left\{\widetilde{D} \in {\left(\bigcup_{1 \leq m \leq k-1} \mathfrak{C}_{m, 1}^k\right)} \smallsetminus \mathfrak{R}_1^k: \substack{\displaystyle\exists~\text{row such that at least} \\ \\ \displaystyle\text{two entries are non-zero}}\right\}}, \\
    \mathfrak{M}^k &= {\left(\bigcup_{1 \leq m \leq k-1} \mathfrak{C}_{m, 1}^k \smallsetminus {\left(\mathfrak{R}_1^k \cup \mathfrak{R}_{2}^k\right)}\right)} \cup \left\{\mathrm{Id}_k\right\}.
\end{align*}
The rest of the terms in \eqref{higher moment formula: eq (1)} will be denoted as $R_{d, k}^{\text{cong}}$, i.e.,
\[
    \mathbb{E}{\left(\prod_{i = 1}^{k}N_i\right)} = M_{d, k}^{\text{cong}} + R_{d, k}^{\text{cong}}.
\]

Define $\mathfrak{M}_{\alpha, \beta}^{\text{cong}}$, for $(\alpha, \beta)$ a division of $\{1, \ldots, k\}$, similar to the affine case and let $M_{\alpha, \beta}^{\text{cong}}$ denote the cardinality of $\mathfrak{M}_{\alpha, \beta}^{\text{cong}}$.
We can re-write \eqref{eqmdkcong} as:
\begin{equation}\label{eqmdkcongnew}
    M_{d, k}^{\text{cong}} = \sum_{(\alpha, \beta)} \sum_{\widetilde{D} \in \mathfrak{M}_{\alpha, \beta}^{\text{cong}}} \int_{(\mathbb{R}^{d})^m} \prod_{i = 1}^{k} \rho_i {\left(\widetilde{D} \begin{pmatrix}
                \bm{y}_1 \\
                \vdots \\
                \bm{y}_m
    \end{pmatrix}\right)} \d{\bm{y}_1} \cdots \d{\bm{y}_m},
\end{equation}
where the outer sum is over all possible divisions $(\alpha, \beta)$ of $\{1, \ldots, k\}$.

\begin{remark}\label{rementryonly1cong}
    For $q \geq 3$, it follows from the definition of $\widetilde{D}$ and similar arguments as in Remark \ref{rementryonly1} that for $\widetilde{D} \in \mathfrak{M}^k$ that the non-zero entries of $\widetilde{D}$ can only be 1. 
    But for $q = 2$ the non-zero entries can be $\pm 1$.
    In particular, this is the reason why we need the condition that $S_d$ is symmetric for the case when $q=2$ (see the second last equality in \eqref{eq: reason of symmetry} below).
\end{remark}

\begin{lemma}[]\label{thmsodmaintermcong}
    For $q \geq 3$ and for $q = 2$ with $S_t$ being symmetric around the origin, we have
    \begin{equation}\label{eqsodmaintermcong}
        M_{d, k}^{\text{\emph{cong}}} = \sum_{(\alpha, \beta)} M_{\alpha, \beta}^{\text{\emph{cong}}} \prod_{i = 1}^{m} V_{\alpha_i}.
    \end{equation}
\end{lemma}
\begin{proof}
    For $q \geq 3$ the proof of this lemma is identical to that of Lemma \ref{thmsodmainterm}.
    Hence we only focus on the case when $q = 2$.

    Consider any matrix $\widetilde{D} = \left(\widetilde{D}_{ij}\right) \in \mathfrak{M}_{\alpha, \beta}^\text{{cong}}$ and let $\lambda_\ell$ be such that $\widetilde{D}_{\beta_\ell, \lambda_\ell} = 1$, $1 \leq \ell \leq k - m$.
    Then, as $S_t$'s are symmetric and $V_i$'s are increasing, the following calculation finishes the proof
    \begin{equation}\label{eq: reason of symmetry}\begin{split}
        &\int_{(\mathbb{R}^{d})^{m}} \prod_{i = 1}^{k} \rho_i\left(\widetilde{D}\begin{pmatrix}
                \vy_1 \\
                \vdots \\
                \vy_{m}
        \end{pmatrix}\right) \d{\vy_1} \cdots \d{\vy_{m}} 
        = \int_{(\mathbb{R}^{d})^{m}} \prod_{i = 1}^k \rho_i{\left(\sum_{j = 1}^{m} \widetilde{D_{ij}} \vy_j\right)} \d{\vy_1} \cdots \d{\vy_{m}} \\
=& \int_{(\mathbb{R}^{d})^{m}} \prod_{i = 1}^{m} \rho_{\alpha_i}(\vy_i) \prod_{\ell = 1}^{k - m} \rho_{\beta_\ell}(\pm\bm{y}_{\lambda_\ell}) \d{\vy_1} \cdots \d{\vy_{m}} 
= \int_{(\mathbb{R}^{d})^m} \prod_{i = 1}^{m} \rho_{\alpha_i}(\bm{y}_i) = \prod_{i = 1}^{m} V_{\alpha_i}.
    \end{split}\end{equation}
\end{proof}

From Lemma~\ref{thmsodmaintermcong} and Lemma~\ref{lemestimatesfromrogers}, we find that:
\begin{theorem}[]\label{thmexpniconvcong}
    \begin{equation}\label{eqexpniconvcong}
        \mathbb{E}{\left(\prod_{i = 1}^{k} N_i\right)} \to \sum_{(\alpha, \beta)} M_{\alpha, \beta}^{\text{\emph{cong}}} \prod_{i = 1}^{m} V_{\alpha_i}.
    \end{equation}
\end{theorem}

\subsubsection{Proof of Theorem \ref{thmppmaincong}}
For $q \geq 3$ the proof of Theorem \ref{thmppmaincong} follows the proof of Theorem \ref{thmppmainaffine}.
We need a small modification in Lemma \ref{lemsodlem3} for the case $q = 2$ because in this case the entries of matrices in $\mathfrak{M}^k$ can be negative.
From now on we only focus on the case $q = 2$ unless otherwise mentioned.

Let $\mathfrak{M}_{\alpha, \beta, +}^{\text{cong}}$ denote the subset of $\mathfrak{M}_{\alpha, \beta}^{\text{cong}}$ of matrices with positive entries and similarly let $\mathfrak{M}_+^k$ denote the subset of $\mathfrak{M}^k$ of matrices with positive entries.
With $M_{\alpha, \beta, +}^{\text{cong}} := \#(\mathfrak{M}_{\alpha, \beta, +}^{\text{cong}})$ note that
\[
    M_{\alpha, \beta}^{\text{cong}} = \#{\left(\mathfrak{M}_{\alpha, \beta}^{\text{{cong}}}\right)} = 2^{k - \#\alpha} M_{\alpha, \beta, +}^{\text{{cong}}}.
\]
Thus from \eqref{eqexpniconvcong} we find
\begin{equation}\label{eqexpniconvcong'}
    \mathbb{E}{\left(\prod_{i = 1}^{k} \widetilde{N}_i\right)} \to \sum_{(\alpha, \beta)} 2^{-\#\alpha} M_{\alpha, \beta, +}^{\text{{cong}}} \prod_{i = 1}^{m} V_{\alpha_i},
\end{equation}
where $\widetilde{N}_i = \frac{1}{2}N_i$ for $1 \leq i \leq k$, i.e., $\widetilde{N}_t = \frac{1}{2}N_t$.

With the following lemma we can compare the RHS's of Theorem \ref{eqexpniconvcong'} and \eqref{eqexpofpp} for $\lambda = 1/2$.
\begin{lemma}[]\label{lemsodlem3i}
    There is bijection $g : \mathfrak{M}_+^k \to \mathscr{P}(k)$ with the property that if $\widetilde{D} \in \mathfrak{M}_+^k$ is an $k \times m$ matrix and $g(\widetilde{D}) = P = \{B_1, \ldots, B_{\#P}\}$ then $\#P = m$ and $\{\alpha_1 < \cdots < \alpha_m\} = \{i_{B_1} < \cdots < i_{B_m}\}$.
\end{lemma}
\begin{proof}
    The same argument with Lemma \ref{lemsodlem3} holds.
\end{proof}

\eqref{eqexpofpp}, \eqref{eqexpniconvcong'} and Lemma \ref{lemsodlem3i} combine to show:
\begin{theorem}[]\label{thmexpniconvcongi}
    For $q = 2$
    \[
        \mathbb{E}{\left(\prod_{i = 1}^{k} \widetilde{N}_i\right)} \to \mathbb{E}{\left(\prod_{i = 1}^{k} N^{1/2}(V_i)\right)}.
    \]
\end{theorem}
 
\begin{corollary}[]\label{corconvindistcong}
    Let $q = 2$, $\bm{V} = (V_1, \ldots, V_k)$ and consider the random vectors
    \[
        \bm{\widetilde{N}}(\Lambda, \bm{V}) = {\left(\widetilde{N}_1(\Lambda), \ldots, \widetilde{N}_k(\Lambda)\right)}
    \]
    and
    \[
        \bm{N}(\bm{V}) = {\left(N^{1/2}(V_1), \ldots, N^{1/2}(V_k)\right)}.
    \]
    Then $\bm{\widetilde{N}}(\Lambda, \bm{V})$ converges in distribution to $\bm{N}(\bm{V})$ as $d \to \infty$.
\end{corollary}
\begin{proof}
    This proof follows similar line of argument as \cite[Corollary 1]{Sod}.
    We omit it for the sake of brevity.
\end{proof}

Corollary \ref{corconvindistcong} implies that all finite dimensional distributions coming from the process $\{\widetilde{N}_t(\Lambda) : t \geq 0\}$ converge to the corresponding finite dimensional distributions of the Poisson process $\{N^{1/2}(t) : t \geq 0\}$ as $d \to \infty$.
By \cite[Theorem 12.6 and Theorem 16.7]{Bil}, we see that the process $\{N_t(\Lambda) : t \geq 0\}$ converges weakly to the process $\{N^{1/2}(t) : t \geq 0\}$ as $d \to \infty$. \\

Corollary \ref{corconvindistcong}, with $k = 1$, is a generalization of \cite[Theorem 3]{Rogers56-2} to the congruence case.

\section{New Moment Formulae}
In this section, we want to simplify Theorem~\ref{higher moment formula: affine} and Theorem~\ref{higher moment formula} for the special case as considered by Str\"{o}mbergsson and S\"{o}dergren in \cite{StSo}. Theorems \ref{New Rogers: affine} and \ref{New Rogers: congruence} below will be used in section \ref{section Limit in distribution}.

For bounded and compactly supported functions $f_i: \RR^d \rightarrow \RR_{\ge 0}$ ($1\le i \le k$), define 
\begin{equation}\label{eq: F}
F_i(\vv)=f_i(\vv) -\int_{\RR^d} f_i \d \vv.
\end{equation}
We want to compute the integrals of $\Siegel{k}{\prod_{i=1}^k F_i}=\prod_{i=1}^k \widehat {F_i}$ over $Y$ and $Y_{\vp/q}$.

We first observe that by applying Theorem~\ref{higher moment formula: affine},
\begin{equation}\label{eq:new}\begin{split}
&\int_{Y} {\prod_{i=1}^k \widehat F_i}(\Lambda) \d\mu_Y (\Lambda)
=\int_{Y} \prod_{i=1}^k \left(\widehat{f_i}(\Lambda) - \int_{\RR^d} f_i \d \vv \right)\\
&=\sum_{A\subseteq \{1, \ldots, k\}} (-1)^a \left(\prod_{i''\in A}\int_{\RR^d} f_{i''} \d \vv\right)\int_{Y} {\prod_{i\in A^c} \widehat{f_i}}(\Lambda) \d\mu_Y(\Lambda) \\
&=\sum_{A\subseteq \{1, \ldots, k\}} (-1)^a \prod_{i''\in A}\int_{\RR^d} f_{i''} \d \vv\;\times\\
&\left(\sum_{m=1}^{k-a} \sum_{u\in \NN} \sum_{\widetilde {D} \in \mathfrak A^{k-a}_{m,u}}
\frac {N( \widetilde{D},u)^d} {u^{dm}}\int_{(\RR^d)^m} \left(\prod_{i \in A^c}f_i\right)\left(\frac {\widetilde{D}}{u}\left(\begin{array}{c} 
\vw_1\\
\vdots\\
\vw_m \end{array}\right)\right) \d\vw_1 \cdots \d\vw_m
\right),
\end{split}\end{equation}
where $a=\# A$ and $A^c=\{1,\ldots,k\}-A$. 

Note that for a given $A\subseteq \{1, \ldots, k\}$ and $\widetilde{D}\in \mathfrak A^{k-a}_{m,u}$, one can find a unique matrix $D''=D''(A, \widetilde{D})\in \mathfrak D^k_{m+a,u}$ (in fact, $\mathfrak A^k_{m+a,u}$) for which
\begin{equation}\label{relation in D''}
\begin{split}
&\left(\prod_{i'' \in A} \int_{\RR^d} f_{i''}\: \d \vv\right)\cdot
\int_{(\RR^d)^m} \left(\prod_{i \in A^c} f_i\right) \left(\frac {\widetilde{D}} u\left(\begin{array}{c} 
\vw_1\\
\vdots\\
\vw_m \end{array}\right)\right) \d\vw_1 \cdots \d\vw_m\\
&\hspace{1in}=\int_{(\RR^d)^{m+a}} \left(\prod_{i=1}^k f_i\right) \left(\frac {D''} u \left(\begin{array}{c}
\vw_1 \\
\vdots \\
\vw_{m+a}\end{array}\right)\right)\d\vw_1 \cdots \d\vw_{m+a}.
\end{split}\end{equation}
Moreover, from the definitions of $N(D'', u)$ and $N(\widetilde{D}, u)$ in Notation~\ref{notation} (3), one can directly obtain the following equality.
\begin{equation}\label{D'' tilde D}
\frac {N(D'',u)^d} {u^{dn}}=\frac {N(\widetilde{D},u)^d}{u^{dm}}.
\end{equation}

We claim the following.

\begin{theorem}\label{New Rogers: affine}
    For $1\le i\le k$, let $F_i$ be the function defined as in \eqref{eq: F} for a bounded and compactly supported function $f_i:\RR^d\rightarrow \RR_{\ge 0}$ ($1\le i\le k$). It follows that
\begin{equation}\label{eq 0: New Rogers affine}
\begin{split}
\int_{Y} {\prod_{i=1}^k \widehat{F_i}}(\Lambda) \d\mu_Y (\Lambda)
=\sum_{n=1}^{k-1} \sum_{u\in \NN} \sum_{D''\in \mathfrak S^k_{n,u}}  
\frac {N(D'',u)^d} {u^{dn}}\int_{(\RR^d)^n} \prod_{i=1}^k f_i \left(\frac {D''} {u} \left(\begin{array}{c}
\vw_1 \\
\vdots \\
\vw_n\end{array}\right)\right)\d\vw_1 \cdots \d\vw_n,
\end{split}\end{equation}
where $\mathfrak S^k_{n,u}\subseteq \mathfrak A^k_{n,u}$ is the set of $D''$ which is one of the following:
\begin{enumerate}[(a)]
\item Each column of $[D'']$ has at least two nonzero elements.
\item There are $0\le a\le n-2$ and $D\in \mathfrak D^{k-a-1}_{n-a-1,u}-\mathfrak A^{k-a-1}_{n-a-1,u}$ for which 
\[
D''=\left(\begin{array}{c|c}
u\Id_a & \\
\hline
 & \begin{array}{c|c}
 u & 0 \cdots 0\\
 \hline
\begin{array}{c}
0 \\
\vdots \\
0 \end{array} & D \end{array}
\end{array}\right),
\]
where each column of $D$ has at least two nonzero elements. 
\end{enumerate}
\end{theorem}

Similarly, from Theorem~\ref{higher moment formula}, we have that $\int_{Y_{\vp/q}} \prod_{i=1}^k \widehat{F_i} (\Lambda) d\mu_q$ is the sum of integrals given as in \eqref{eq:new} with replacing $\mathfrak A^{k-a}_{m,u}$ by $\mathfrak C^{k-a}_{m,u}$. 
For $D''=D''(A,\widetilde D)\in \mathfrak D^k_{m+a,u}$ defined using $A\subseteq \{1, \ldots, k\}$ and $\widetilde D\in \mathfrak C^{k-a}_{m,u}$ as in \eqref{relation in D''}, we will see that $D''\in \mathfrak C^k_{m+a,u}$. 
It is easily seen that the equality \eqref{D'' tilde D} holds in the congruence case, either.

\begin{theorem}\label{New Rogers: congruence}
For $1\le i\le k$, let $F_i$ be the function defined as in \eqref{eq: F} for a bounded and compactly supported function $f_i:\RR^d\rightarrow \RR_{\ge 0}$ ($1\le i\le k$). It follows that
\[\begin{split}
&\int_{Y_{\vp/q}} {\prod_{i=1}^k \widehat F_i}(\Lambda) \d\mu_q (\Lambda)\\
&\hspace{0.4in}=\sum_{n=1}^{k-1} \sum_{u\in \NN} \sum_{D''\in \mathfrak T^k_{n,u}} \frac {N(D'',u)^d} {u^{dn}}
\int_{(\RR^d)^n} \prod_{i=1}^k f_i \left(\frac {D''} {u}  \left(\begin{array}{c}
\vw_1 \\
\vdots \\
\vw_n\end{array}\right)\right)\d\vw_1 \cdots \d\vw_n,
\end{split}\]
where $\mathfrak T^k_{n,u}$ is a subset of $\mathfrak C^k_{n,u}$ collecting $D''$ which is one of the following:
\begin{enumerate}[(a)]
\item Each column of $D''$ has at least two nonzero elements.
\item There are $0\le a\le n-2$ and $\widetilde D\in \mathfrak C^{k-a}_{n-a,u}$ so that 
\[
D''=\left(\begin{array}{cc}
u\Id_a & \\
 & \widetilde{D} \end{array}
\right),
\]
where $[\widetilde{D}]^1={\tp{(u, 0, \ldots, 0)}}$ and any other columns of $\widetilde{D}$ have at least two nonzero elements.
Moreover, the right-bottom minor of $[\widetilde{D}]$ with size $(k-a-1)\times(n-a-1)$ is not an element of $\mathfrak C^{k-a-1}_{n-a-1,u}$ (or any $\mathfrak C^{k-a-1}_{n-a-1,*}$).
\end{enumerate}
\end{theorem}
\begin{proof}[Proof of Theorem~\ref{New Rogers: affine} and Theorem~\ref{New Rogers: congruence}]
As described in \eqref{relation in D''}, a possible matrix $D''$ among elements of $\mathfrak D^k_{n,u}$ is constructed by using $A\subseteq \{1, \ldots, k\}$ and $\widetilde{D}\in \mathfrak A^{k-a}_{n-a,u}$.
Conversely, we want to consider all possible pairs $(A, \widetilde{D})$ which give the same $D''$.

Let such a $D''=(d''_{ij})$ be given. Denote
\[
B=\left\{1\le i''\le k :
\begin{array}{l}
\text{$1\le \exists j_0\le n$ for which}\\
\hspace{0.3in}\text{$d''_{i'' j}=0$ for all $j$ except $d''_{i'' j_0}=u$ and}\\[0.05in]
\hspace{0.3in}\text{$d''_{i j_0}=0$ for all $i$ except $d''_{i'' j_0}=u$}
\end{array}\right\}.
\]
After changing the (last $(k-b_1)$) coordinates of $\RR^k$, we may assume that
\begin{equation}\label{D''}
\frac {D''} {u}=\left(\begin{array}{ccc}
\Id_{b_1} & &  \\ 
& \begin{array}{c}
\dfrac {\widetilde {D_0}} {u} \end{array}& \\ 
& & \Id_{b_2}
\end{array}\right),
\end{equation}
where $b_1$ and $b_2$ could be $0$ (then $D''/u$ will be one- or two-block diagonal matrix) and $\widetilde D_0\in \mathfrak A^{k-b_1-b_2}_{n-b_1-b_2,u}$ (or $\mathfrak C^{k-b_1-b_2}_{n-b_1-b_2,u}$, respectively) is the minimal size among possible $(A,\widetilde{D})$ for which $D''(A,\widetilde{D})=D''$, i.e., 
\begin{center}
    each column of $\widetilde D_0$ except $[\widetilde{D_0}]^1$ has at least two nonzero elements.
\end{center}

Notice that any matrix constructed by choosing more than $k-b_1-b_2$ rows and more than $n-b_1-b_2$ columns from $D''/u$ and having $\widetilde{D_0}/u$ as its minor is element of $\mathfrak A^*_{*,u}$ (or $\mathfrak C^*_{*,u}$, respectively).
For example, $D''\in \mathfrak C^k_{n,u}$ since $D''$ is constructed by $(\overline{D},t,{\tp{(0,\ldots, 0, \vl, 0, \ldots, 0)}})$ under the map in \eqref{transition: cong}, where $\overline{D}$ is the right-bottom minor of $uD''$ with size $(k-1)\times (n-1)$, and $(t, \vl)$ is a pair used for defining $\widetilde D$.

Now let us check case by case.
Denote by 
\[B_1=\{i\in B: i \le b_1+1\}
\quad\text{and}\quad 
B_2=\{k-b_2+1, \ldots, k\}\] 
so that $B=B_1\cup B_2$. Note that $(b_1+1)$ could be not contained in $B$.
Observe that possible $A$ for constructing $D''$ is of the form $A_1 \cup A_2$, where $A_1 \subseteq B_1$ and $A_2 \subseteq B_2$. The difference between $A_1$ and $A_2$ is that $A_1$ may have an extra condition according to the given $D''$, but any subset of $B_2$ can be $A_2$. 

\vspace{0.1in}
We first assume that $B_2\neq \emptyset$.
Since
\[\begin{split}
\sum_{\scriptsize \begin{array}{c}
\text{``possible''}\\
 A\subseteq B\end{array}} (-1)^{\# A}
&=\sum_{\scriptsize \begin{array}{c}
\text{``possible''}\\
A_1 \subseteq B_1\end{array}} (-1)^{\# A_1} \sum_{\forall A_2 \subseteq B_2} (-1)^{\# A_2}\\
&=\sum_{\scriptsize \begin{array}{c}
\text{``possible''}\\
A_1 \subseteq B_1\end{array}} (-1)^{\# A_1} \cdot 0 = 0,
\end{split}\]
with the observation in \eqref{relation in D''}, the partial sum 
\begin{equation}\begin{split}\label{partial sum}
&\sum_{\scriptsize \begin{array}{c}
A, \widetilde{D}\text{ s.t.}\\
D''(A, \widetilde{D})=D''\end{array}}\hspace{-0.15in}
(-1)^{\# A}\:\frac {N(D'',u)^d} {u^{dr}}
\int_{(\RR^d)^n} \prod_{i=1}^k f_i \left(\frac {D''} u \left(\begin{array}{c}
\vw_1 \\
\vdots \\
\vw_{m+a}\end{array}\right)\right)\d\vw_1\cdots \d\vw_{m+a}
\end{split}\end{equation}
associated with $D''$ in the right hand side of \eqref{eq 0: New Rogers affine} is zero.

\vspace{0.1in}
Now let us assume that $b_2=0$. 
If $B=\emptyset$, that is, 
\begin{center}
    Each column of $D''$ has at least two nonzero vectors,
\end{center}
and only possible $(A, \widetilde{D})$ is $(\emptyset, D'')$. This is the case (a) in the theorem.

Suppose that $|B|=b_1\ge 1$. Equivalently, suppose that  
\begin{center}
    $[\widetilde {D_0}]$ as well as other columns of $\widetilde D_0$ has at least two nonzero elements.
\end{center}
Then any subset $A$ of $B$ is possible for defining $D''$, hence the partial sum \eqref{partial sum} is zero.

The only left case is when $|B|=b_1+1\ge 2$. Notice that 
\begin{center}
    $[\widetilde{D_0}]^1={\tp{(u,0,\ldots,0)}}$, and\\
    the right-bottom minor of $\widetilde{D_0}$ with size $(k-b_1-1)\times (n-b_1-1)$\\ 
    is not an element of $\mathfrak A^{k-b_1-1}_{n-b_1-1, u}$ ($\mathfrak C^{k-b_1-1}_{n-b_1-1, u}$, respectively).
\end{center}
In this case, any subset of $B$ except $B$ itself can be possible A for defining $D''$, and this is the case (b) in the theorem.
\end{proof}

\section{CLT and Brownian motion}\label{section Limit in distribution}

As in Section~\ref{Poisson}, we will use the method of moments which is applicable with the normal distribution and Brownian motion, following \cite{StSo}.
Recall that the $k$-th moment of the normal distribution is $0$ when $k$ is odd and $(k-1)!!$ when $k$ is even.

For Brownian motion, it suffices to show that the induced measure $P^1_d$ and $P^{\vp/q}_d$ from $Z^1_d(t)$ and $Z^{\vp/q}_d(t)$ respectively, on the space $C[0,1]$ of continuous real-valued functions on $[0,1]$, weakly converge to Wiener measure as $d$ goes to infinity. 

\vspace{0.1in}
Let $\phi:\NN \rightarrow \RR_{>0}$ be a function for which $\lim_{d\rightarrow \infty} \phi(d)=\infty$ and $\phi(d)=O_{\varepsilon}(e^{\varepsilon d})$ for every $\varepsilon>0$.
Let $\iota\in \NN$ and $c_1, \ldots, c_\iota>0$ be arbitrarily given. For each $d\in \NN$, consider $S_{i,d}\in \RR^{d}$ be a Borel measurable set satisfying $\vol(S_{i,d})=c_i \phi(d)$ for $1\le i\le \iota$ and $S_{i,d}\cap S_{i',d}=\emptyset$ if $i\neq i'$.
If we consider the case that $\Lambda \in Y_{{\vp}/q}$ with $q=2$, we further assume that $S_{i,d}=-S_{i,d}$ for $1\le i\le \iota$ and $d\in \NN$.

Let 
\[\begin{split}
Z^1_{i,d}&:= \frac {\#(\Lambda \cap S_{i,d})-c_i\phi(d)}{\sqrt{\phi(d)}},\hspace{1.25in}\Lambda\in Y\;\text{and}\\[0.1in]
Z^{\vp/q}_{i,d}&:=\left\{\begin{array}{cl}
\dfrac {\#(\Lambda \cap S_{i,d})-c_i\phi(d)}{\sqrt{\phi(d)}}, &\text{if } q\neq 2;\\[0.2in]
\dfrac {\#(\Lambda \cap S_{i,d})-c_i\phi(d)}{\sqrt{2\phi(d)}}, &\text{if } q=2,
\end{array}  
\right.\quad \Lambda \in Y^{\vp/q}.
\end{split}\]

\begin{proposition}\label{SS Proposition 4.1} 
Let $\ds=1$ or $\vp/q$.
For any fixed $\vk=(k_1, \ldots, k_\iota)\in \NN^\iota$, it follows that
\[\begin{split}
&\lim_{d\rightarrow \infty} \EE\left((Z^\ds_{1,d})^{k_1} \cdots (Z^\ds_{\iota,d})^{k_\iota}\right)\\
&\hspace{0.5in}=\left\{\begin{array}{cl}
\prod_{i=1}^\iota c_i^{k_i/2} (k_i-1)!!, &\text{if } k_1, \ldots, k_\iota \text{ are all even},\\[0.05in]
0, &\text{ otherwise.}\end{array}\right.
\end{split}\]
\end{proposition}
\begin{proof} 
Let $k=k_1+\cdots+k_\iota$ and consider $d> \lfloor k^2/4\rfloor +3$. 
For each $d\in \NN$ and $1\le i\le \iota$, let $f_{i,d}$ be the indicator function of $S_{i,d}$ and define 
$$F_{i,d}(\Lambda)=\widehat{f_{i,d}}(\Lambda) -\int_{\RR^d} f_{i,d} \d \vv=\widehat{f_{i,d}}(\Lambda)- c_i\phi(d), \;\Lambda \in Y.$$
We will use Theorem~\ref{New Rogers: affine} and Theorem~\ref{New Rogers: congruence}.

Recall that we can divide 
$\bigcup_{1\le n \le k}
\bigcup_{u \in \NN}\mathfrak D^k_{n,u}$
as the union of $\mathfrak R^k_{1}$, $\mathfrak R^k_{2}$ and $\mathfrak M^k$, where
\[\begin{split}
\mathfrak R^k_{1}&=
\bigcup_{1\le n \le k-1}\left(\left(\bigcup_{u\ge 2} \mathfrak D^k_{n,u}\right) \cup
\left\{D=(d_{ij}) \in \mathfrak D^k_{n,1} : \exists |d_{ij}| \ge 2 \right\}\right);\\
\mathfrak R^k_{2}&=
\left\{D\in \Big(\bigcup_{1\le n \le k-1}\mathfrak D^k_{n,1}\Big) - \mathfrak R^k_{1} : \begin{array}{l}
\text{$\exists$ column such that}\\
\text{at least two entries are nonzero}\end{array}\right\};\\
\mathfrak M^k&=
\Big(\bigcup_{1\le n\le k-1}\mathfrak D^k_{n,1}\Big) - \left(\mathfrak R^k_{1} \cup \mathfrak R^k_{2}\right).
\end{split}\]

\vspace{0.15in}
\noindent (i) The space $Y$:

By Theorem~\ref{New Rogers: affine} and Theorem~\ref{lemestimatesfromrogers}, one can deduce that
\begin{equation}\label{eq 2: SS Proposition 4.1}
\begin{split}
&\EE \left(\prod_{i=1}^\iota (Z^1_{i,d})^{k_i}\right)
=\frac 1 {\phi(d)^{k/2}} \int_{Y} {\prod_{i=1}^{\iota} \widehat F_{i,d}^{k_i}} (\Lambda)\d\mu_Y(\Lambda)\\
&=\frac 1 {\phi(d)^{k/2}}\sum_{n=1}^{k-1}
\sum_{\scriptsize \begin{array}{c}
D'' \in\\
 \mathfrak S^k_{n,1}\cap \mathfrak M^k\end{array}}
\int_{(\RR^d)^n} \left(\prod_{i=1}^{\iota} f_{i,d}^{k_i}\right) \left(D'' \left(\begin{array}{c} \vw_1 \\ \vdots \\ \vw_n \end{array}\right)\right)\d\vw_1 \cdots \d\vw_n \\
&\hspace{0.2in}+ O\left(\left(\frac {\sqrt{3}} 2\right)^d \phi(d)^{k/2-1}\right).
\end{split}\end{equation}
Notice that if $D'' \in \mathfrak S^k_{n,1}\cap \mathfrak M^k$, then $D''$ is of type (a) in Theorem~\ref{New Rogers: affine}. Hence, for each column of $D''$, there are at least two nonzero entries and for each row of $D''$, there is exactly one nonzero entry.
Moreover, as mentioned in Remark~\ref{rementryonly1}, entries of $D''$ is $\{0,1\}$.

We first claim that $D''$ for which the inner integral above is nontrivial is the block diagonal matrix of the form
\[
\left(\begin{array}{cccc}
D''_{k_1, n_1} & & & \\
& D''_{k_2, n_2} & & \\
& & \ddots & \\
& & & D''_{k_\iota, n_\iota} 
\end{array}\right),
\]
where $n_1+\cdots+n_\iota=n$ and $n_i\ge 1$ for each $1\le i\le \iota$. Moreover, each $D''_{k_i,n_i} \in \mathfrak S^{k_i}_{n_i,1}\cap \mathfrak M^{k_i}$.
Indeed, since the set $\{S_{i,n}\}_{1\le i\le \iota}$ is mutually disjoint, for each column, it is only possible that nontrivial entries are located between the $\left((\sum_{\ell=1}^{i-1} k_i)+1\right)$-th row and the $\left(\sum_{\ell=1}^{i} k_i\right)$-th row for some $1\le i \le \iota$, in other words, nontrivial entries are concentrated in rows which correspond to $f_i$.
The fact that $D''$ is a block diagonal matrix comes from that $D''\in \mathfrak D^k_{n,u}$, especially, from the first property of Notation~\ref{notation} (2). 

It is not hard to show that each $D''_{k_i,n_i}$ is in $\mathfrak S^{k_i}_{n_i,1} \cap \mathfrak M^{k_i}$ from the fact that $D'' \in \mathfrak S^k_{n,1}\cap \mathfrak M^k$.
Hence the main term in \eqref{eq 2: SS Proposition 4.1} is
\begin{equation}\label{eq 3: SS Proposition 4.1}
\prod_{i=1}^\iota \frac 1 {\phi(d)^{k_i/2}} 
\hspace{-0.05in}\sum_{n_i=1}^{\lfloor k_i/2\rfloor}
\hspace{-0.15in}\sum_{\scriptsize \begin{array}{c}
D''_{k_i,n_i} \hspace{-0.05in}\in\\
\mathfrak S^{k_i}_{n_i,1} \cap \mathfrak M^{k_i}\end{array}}
\hspace{-0.2in}
\int_{(\RR^d)^{n_i}} f_i^{k_i}\left(D''_{k_i,n_i}\left(\begin{array}{c} \vw_1 \\ \vdots \\ \vw_{n_i}\end{array}\right)\right) \d\vw_1 \cdots \d\vw_{n_i}.
\end{equation}
The next claim is that for each $i$, there is a one-to-one correspondence between $\mathfrak S^{k_i}_{n_i,1} \cap \mathfrak M^{k_i}$ and the set of partitions $\mathcal P=\{P_1, \ldots, P_{n_i}\}$ of $\{1, \ldots, k_i\}$ such that
\[
|\mathcal P|=n_i 
\text{ and }
|P_\ell|\ge 2 \text{ for } 1\le \ell \le n_i. 
\] 
Let $\mathcal P$ be such a partition.
Reordering if necessary, we may assume that $\min P_1 < \ldots < \min P_{n_i}$. The corresponding element in $\mathfrak S^{k_i}_{n_i,1} \cap \mathfrak M^{k_i}$ is 
\begin{equation}\label{eq 1: SS Proposition 4.1}
\left[D''_{k_i,n_i}\right]_{\ell j}=\left\{\begin{array}{cl}
1, &\text{if } \ell \in P_j;\\
0, &\text{otherwise.}
\end{array}\right.
\end{equation}
It is obvious that from the first property of Notation~\ref{notation} (2) and the definition of $\mathfrak M^{k_i}$, any element in $\mathfrak S^{k_i}_{n_i,1} \cap \mathfrak M^{k_i}$ is a matrix of the form \eqref{eq 1: SS Proposition 4.1} for some partition $\{P_1, \ldots, P_{n_i}\}$ of $\{1, \ldots, k_i\}$.

Let $N(k_i,n_i)$ be the number of such partitions. If $n_i< k_i/2$, since $\lim_{d\rightarrow \infty} \phi(d)=\infty$,
\begin{equation}\label{eq 4: SS Proposition 4.1}
\begin{split}
&\frac 1 {\phi(d)^{k_i/2}} 
\sum_{\scriptsize \begin{array}{c}
D''_{k_i,n_i}\hspace{-0.05in}\in\\
\mathcal S^{k_i}_{n_i,1}\cap \mathfrak M^{k_i}\end{array}}
\int_{(\RR^d)^{n_i}} F_i^{k_i} \left(D''_{k_i,n_i}\left(\begin{array}{c}
\vw_1 \\ \vdots \\ \vw_{n_i}\end{array}\right)\right) \d\vw_1 \cdots \d\vw_{n_i}\\
&\hspace{1in}\le \frac {c_i^{n_i}N(k_i,n_i)} {\phi(d)^{k_i/2-n_i}} 
\longrightarrow 0\;\text{as }d \rightarrow \infty.
\end{split}\end{equation}

If $n_i=k_i/2$, by the induction on $k_i/2$, one can show that
\begin{equation}\label{eq 5: SS Proposition 4.1}\begin{split}
&\frac 1 {\phi(d)^{k_i/2}} 
\sum_{\scriptsize \begin{array}{c}
D''_{k_i,n_i}\hspace{-0.05in}\in \\
\mathcal S^{k_i}_{n_i,1}\cap \mathfrak M^{k_i}\end{array}}
\int_{(\RR^d)^{n_i}} F_i^{k_i} \left(D''_{k_i,n_i}\left(\begin{array}{c}
\vw_1 \\ \vdots \\ \vw_{n_i}\end{array}\right)\right) \d\vw_1 \cdots \d\vw_{n_i}\\
&\hspace{1in}=c_i^{k_i/2} N(k_i, k_i/2)= c_i^{k_i/2} (k_i-1)!!.
\end{split}\end{equation}

The result follows from \eqref{eq 3: SS Proposition 4.1}, \eqref{eq 4: SS Proposition 4.1} and \eqref{eq 5: SS Proposition 4.1}.

\vspace{0.15in}
\noindent (ii) The space $Y_{\vp/q}$:

The proof is similar to that of (i), where we use Theorem~\ref{New Rogers: congruence}, Lemma~\ref{lemestimatesfromrogers}.
One can check that $D''\in \mathfrak T^k_{n,1}\cap \mathfrak M^k$ is of type (a) in Theorem~\ref{New Rogers: congruence}.

One difference from the affine case is when $q=2$, $D''\in \mathfrak T^k_{n,1}\cap \mathfrak M^k$, which permits to have $-1$ as its entries. More precisely, the rows corresponding to $I_{D''}^c$ can have $\pm 1$ as their nonzero entries. 

It follows that
\[\begin{split}
&\lim_{d\rightarrow \infty}\EE\left(\prod_{i=1}^\iota Z^{k_i}_{i,d}\right)=\lim_{d\rightarrow \infty}
\prod_{i=1}^\iota \frac 1 {(2\phi(d))^{k_i/2}}\times\\
&\hspace{0.3in} \sum_{n_i=1}^{\lfloor k_i/2 \rfloor}\sum_{\scriptsize \begin{array}{c}
D''_{k_i,n_i}\in\\
\mathfrak T^{k_i}_{n_i,1}\cap \mathfrak M^{k_i}\end{array}} 
\int_{(\RR^d)^{n_i}} f_i^{k_i}\left(D''_{k_i,n_i} \left(\begin{array}{c} \vw_1 \\ \vdots \\ \vw_{n_i}\end{array}\right)\right)\d\vw_1 \cdots \d\vw_{n_i}.
\end{split}\]

As in the affine case, the limit is nontrivial only if all $k_i$'s are even and is determined by summation over $\mathfrak T^{k_i}_{k_i/2,1}\cap \mathfrak M^{k_i}$. Hence if $q=2$, since $\# I_{D''_{k_i,k_i/2}}^c=k_i/2$, the number $\#\left(\mathfrak T^{k_i}_{k_i/2,1}\cap \mathfrak M^{k_i}\right)$ is $2^{k_i/2}N(k_i,k_i/2)$. Therefore

\[\begin{split}
&\prod_{i=1}^\iota\frac 1 {(2\phi(d))^{k_i/2}} 
\sum_{\scriptsize \begin{array}{c}
D''_{k_i,n_i}\hspace{-0.05in}\in \\
\mathfrak T^{k_i}_{n_i,1}\cap \mathfrak M^{k_i}\end{array}}
\int_{(\RR^d)^{n_i}} F_i^{k_i} \left(D''_{k_i,n_i}\left(\begin{array}{c}
\vw_1 \\ \vdots \\ \vw_{n_i}\end{array}\right)\right) \d\vw_1 \cdots \d\vw_{n_i}\\
&=\prod_{i=1}^\iota \frac 1 {(2\phi(d))^{k_i/2}}
(c_i\phi(d))^{k_i/2} \cdot 2^{k_i/2}N(k_i, k_i/2)
=\prod_{i=1}^\iota c_i^{k_i/2} (k_i-1)!!.
\end{split}\]
\end{proof}

\begin{proof}[Proofs of Theorem~\ref{Strom-Sod main thm: affine} and ~\ref{Strom-Sod main thm: congruence}]

As a corollary of Proposition~\ref{SS Proposition 4.1} with $\iota=1$, for $\ds=1$ and $\vp/q$, it follows that for any $k\in \NN$,
\[
\lim_{d\rightarrow \infty} \EE\left( (Z^\ds_d)^k\right)
=\left\{\begin{array}{cl}
(k-1)!!, &\text{if } k\in 2\NN;\\
0, &\text{otherwise,}
\end{array}\right.
\]
which shows that $Z^\ds_d\rightarrow \mathcal N(0,1)$ as $d\rightarrow \infty$ in distribution by the method of moments.
\end{proof}

\begin{proof}[Proofs of Theorem~\ref{Strom-Sod Brownian: affine} and ~\ref{Strom-Sod Brownian: congruence}]
For any $0<t_1<\ldots<t_\iota<1$, set
\[
S_{i,d}=(t_i)^{1/d}S_d - (t_{i-1})^{1/d}S_d,\; 2\le i\le \iota
\]
and $S_{1,d}=(t_1)^{1/d}S_d$. Since $S_d$ is star-shaped, all $S_{i,d}$'s are mutually disjoint.
By Proposition~\ref{SS Proposition 4.1}, for $\ds=1$ and $\vp/q$, the random vector
\[
\left(Z^\ds_d (t_1), Z^\ds_d(t_2)- Z^\ds_d(t_1), \ldots, Z^\ds_d(t_\iota)-Z^\ds_d(t_{\iota-1})\right)
\]
converges weakly as finite-dimensional distributions to
\[
\left(\mathcal N(0,t_1), \mathcal N(0, t_2-t_1), \ldots, \mathcal N(0, t_\iota-t_{\iota-1})\right)
\]
by the method of moments.

\vspace{0.1in}
The rest of the proof is to show the tightness.
As in the proof of Theorem 1.6 in \cite{StSo}, by \cite[Theorem 13.3 and (13.14)]{Bil}, it suffices to show that for any $0\le r \le s \le t \le 1$, 
\begin{equation*}\label{eq: tightness}
\EE\left((Z^\ds_d(s)-Z^\ds_d(r))^2(Z^\ds_d(t)-Z^\ds_d(s))^2\right)
\ll (\sqrt t - \sqrt r)^2.
\end{equation*}

We omit the proof since it is almost the same as in the proof of \cite[Theorem 1.6]{StSo} (see especially equations from (4.4) to (4.9)), where the arguments are appliable to a star-shaped set $S_d\subseteq \RR^d$ centered at the origin without any modification.
Here, we want to remark that we need the argument in \cite{StSo} only for the congruence case. 
For the affine case, since $\bigcup_{u\in \NN} \mathfrak S^4_{1,u} \cap (\mathfrak R_1 \cup \mathfrak R_2)=\emptyset$, it deduced directly from (4.5) in \cite{StSo} that
\[\begin{split}
&\EE\left((Z_d(s)-Z_d(r))^2(Z_d(t)-Z_d(s))^2\right)\\
&\ll (t-r)^2 + \max\left(\left(\frac 3 4\right)^{n/2}(t-r)^2, \left(\frac 3 4\right)^{n/2}(t-r)^3 \phi(d)\right)\\
&\ll (t-r)^2 < (\sqrt t - \sqrt r)^2.
\end{split}\] 
\end{proof}

\subsection*{Conflict of interest}
On behalf of all authors, the corresponding author states that there is no conflict of interest.



\end{document}